\documentclass{amsart}
\usepackage[margin=1.5in]{geometry}
\usepackage{amsthm}
\usepackage{amssymb}
\usepackage[non-compressed-cites]{amsrefs}
\usepackage{esint}
\usepackage{caption}
\usepackage{enumitem}
\usepackage{mathtools}
\usepackage{xcolor}
\usepackage{microtype}
\usepackage{multirow}
\usepackage{array}
\usepackage{tikz}
\usepackage{threeparttable}
\usepackage{hyperref}

\usetikzlibrary{arrows.meta,decorations.markings,calc}

\allowdisplaybreaks

\newcommand{\R}{\mathbb{R}}

\newcommand{\vol}{\operatorname{vol}}
\newcommand{\dvol}{\mathrm{dvol}}

\newcommand{\Z}{\mathbb{Z}}

\newcommand{\C}{\mathbb{C}}
\newcommand{\im}{i}
\newcommand{\id}{\mathbb{I}}

\newcommand\be{\begin{eqnarray}}
\newcommand\ee{\end{eqnarray}}

\newtheorem{theorem}{Theorem}[section]
\newtheorem{corollary}[theorem]{Corollary}

\newtheorem{definition}[theorem]{Definition}
\newtheorem{proposition}[theorem]{Proposition}

\numberwithin{equation}{section}
 
\hypersetup{
    colorlinks,
    linkcolor={red!60!black},
    citecolor={blue!80!black},
}

\setlist[itemize]{leftmargin=1em}
\setlist[enumerate]{leftmargin=1em}
\title[Four-dimensional Riemannian geometry via 2-forms]{Four-dimensional Riemannian geometry via 2-forms}

\author[Bhoja]{Niren Bhoja}

\author[Krasnov]{Kirill Krasnov}
\email{kirill.krasnov@nottingham.ac.uk, ORCID: 0000-0003-2800-3767}
\address{School of Mathematical Sciences, University of Nottingham, Nottingham, NG7 2RD, UK}

\begin{document}

\begin{abstract} In differential geometry, geometric structures can often be encoded by differential forms satisfying algebraic and differential constraints. This is in particular the case for spinorial $G$-structures, where the defining tensors are differential forms arising as spinor bilinears and their exterior derivatives determine the intrinsic torsion. In this paper we show that, in certain situations, this can be extended beyond the setting of spinorial $G$-structures. Thus, when $\tilde{G}/G$ is a Lie group $H$, a $\tilde{G}$-structure with $\tilde{G}\supset G$ can be described in terms of a spinorial $G$-structure by allowing the defining forms to take values in an associated $H$-bundle, and converting the intrinsic torsion of the $G$-structure into an $H$-connection. We develop this idea in four dimensions, where the triple of 2-forms associated with a spinorial ${\rm SU}(2)$-structure can be encoded as a 2-form with values in the associated $H={\rm SO}(4)/{\rm SU}(2)={\rm SO}(3)$ vector bundle. This gives a description of Riemannian geometry, i.e. ${\rm SO}(4)$-structures, and leads to a unique ${\rm SO}(3)$-invariant functional of ${\rm SU}(2)$-structures whose critical points are Einstein. This perspective also provides a unified framework for Riemannian, K\"ahler and hyper-K\"ahler geometries in four dimensions. 
 \end{abstract}
   

\maketitle


\section{Introduction}

Many geometric structures in low dimensions admit descriptions in terms of differential forms. In four dimensions, it is classical that a Riemannian metric can be recovered from a rank-three subbundle of the bundle of two-forms on which the wedge product is positive definite. In this paper we take this observation as a starting point for a systematic treatment of four-dimensional Riemannian geometry in terms of an $\mathrm{SO}(3)$ bundle-valued $\mathrm{SU}(2)$-structure. This leads to a description of the intrinsic torsion of an $\mathrm{SU}(2)$-structure in terms of an $\mathrm{SO}(3)$-connection. A notable feature of this approach is that it singles out, under natural invariance assumptions, a canonical functional on the space of such structures, whose critical points correspond to Einstein metrics. In this way, the Einstein condition emerges from the intrinsic geometry of $\mathrm{SU}(2)$-structures rather than being imposed externally. 

\subsection{G-structures, Intrinsic torsion, Action functionals.} To motivate the constructions that follow, we start with the following standard definition
\begin{definition}
A G-structure on a smooth manifold $M$ is a reduction of the principal ${\rm GL}(n,\R)$ bundle of frames on $M$ to a G-subbundle. 
\end{definition}
Most interesting geometric structures on $M$ can be rephrased in the language of G-structures. For example, a Riemannian metric on $M$ is an ${\rm O}(n)$-structure, an almost complex structure $J: TM\to TM, J^2=-\mathbb{I}$ on a manifold of dimension $2n$ is a ${\rm GL}(n,\C)$-structure, and an almost-Hermitian structure on a manifold of dimension $2n$ is a ${\rm U}(n)$-structure. 

Given a G-structure, one can construct a first-order in derivatives object transforming covariantly under diffeomorphisms called intrinsic (or reduced) torsion, see e.g. \cite{Bryant}, Section 1.3 for the definition. The intrinsic torsion vanishes for an ${\rm O}(n)$-structure (and this is one way to state the fundamental theorem of Riemannian geometry). But in general the intrinsic torsion of a G-structure is non-trivial. Its vanishing is then the most natural first-order in derivatives differential equation to be imposed on a G-structure, see \cite{Chern}. For example, the intrinsic torsion of a ${\rm U}(n)$-structure vanishes if and only if it is K\"ahler.  

For what follows, we are mainly interested in metric G-structures. 
\begin{definition} A G-structure on an n-dimensional manifold $M$ is called metric if $G\subset {\rm O}(n)$. 
\end{definition}
For metric G-structures, a natural class of diffeomoprhism-invariant functionals leading to second-order in derivatives Euler-Lagrange equations can be constructed as the integral of a linear combination of all independent torsion squared terms. More concretely, for a metric G-structure, we define the intrinsic torsion of a $G$-structure to be the torsion of the $\mathfrak{g}$-part of the Levi-Civita connection. Given that the Levi-Civita connection is torsion-free, the intrinsic torsion becomes an object in $T\in\Lambda^1 \otimes \mathfrak{g}^\perp$, where $\mathfrak{so}(n) = \mathfrak{g}\oplus\mathfrak{g}^\perp$. It decomposes into a number of irreducible representations of $G$
\[
T = \sum_k T_k,
\]
and we can consider
\begin{equation}\label{gen-action-intr}
S= \int_M (\sum_k c_k |T_k|^2) \dvol,
\end{equation}
where $c_k$ are arbitrary coefficients (that do not need to be positive), and $\dvol$ is the volume form for the metric defined by the G-structure. The number of possible terms here can be quite large. For example, for ${\rm SU}(2)$-structures in four dimensions, this number is 15, see \cite{Fowdar}, formula (66) of the paper. 

In general, it is not clear if there is a preferred member in the family of action functionals of this type. One of the motivations of this paper is to explain that it is sometimes possible to demand that a functional of this type possesses extra symmetry, and that this can severely cut the number of arbitrary constants in (\ref{gen-action-intr}). In our main example of interest in this paper, with $G={\rm SU}(2)$ in four dimensions, the construction we are to explain leads to a unique (modulo an overall coefficient) functional.  

In general, a G-structure can be encoded into a collection of $G\subset{\rm GL}(n)$-invariant tensors on $M$; in some cases this collection consists of a single tensor. The simplest examples are: an ${\rm O}(n)$-structure is encoded by a Riemannian metric; An ${\rm Sp}(n)$-structure in $2n$ dimensions is encoded by a non-degenerate 2-form, i.e. an almost-symplectic structure. A more exotic but still very well-known example is that of a $G_2$-structure in dimension seven, where the relevant tensor is a generic 3-form (of a positive type). When a functional of the type (\ref{gen-action-intr}) exists, it can be thought of as a diffeomorphism-invariant functional constructed from squares of first derivatives of these tensors. 

\subsection{Spinorial G-structures.} Certain metric G-structures can be called "spinorial". The importance of these structures lies in the fact that their encoding tensors are differential forms, and that the intrinsic torsion is completely determined by the exterior derivatives of these differential forms. 

We recall that, given an oriented spin Riemannian manifold $(M,g)$, a spin structure on $M$ is a lift of the principal ${\rm SO}(n)$ bundle of oriented orthonormal frames to a ${\rm Spin}(n)$ bundle. Let $\lambda: {\rm Spin}(n)\to {\rm SO}(n)$ be the covering map. A spinorial G-structure is then a reduction of the principal ${\rm Spin}(n)$ bundle to a $G'$-subbundle, with $G'={\rm Stab}_\psi\subset {\rm Spin}(n)$ being a subgroup that stabilises a spinor $\psi\in S$. The $G$ of the G-structure arises as $G=\lambda(G')\subset {\rm SO}(n)$. When $n$ is even, the spinor representation $S$ of ${\rm Spin}(n)$ splits into two irreducible components $S=S_+\oplus S_-$, and in this case it is natural to consider spinorial G-structures defined by a semi-spinor $\psi\in S_+$. Given that 
\be
{\rm dim}( {\rm GL}(n,\R)/G) = {\rm dim}( {\rm GL}(n,\R)/ O(n)) + {\rm dim}({\rm Spin}(n)/G),
\ee
we see that the ${\rm GL}(n,\R)$ orbit of spinorial G-structures can be viewed as the space of Riemannian metrics together with a spinor of algebraic type that has $G$ as the stabiliser. Thus, heuristically, we can say
\be
\{ \text{spinorial G-structure}\} = \{\text{metric}\} + \{\text{spinor}\}.
\ee

What is important about the spinorial G-structures is that they can be expected to be encodable into a collection of $G\subset{\rm GL}(n)$-invariant differential forms on $M$. Thus, given a spinor $\psi$ in $S$ or in $S_+$, a certain collection of differential forms arises as spinor bilinears constructed from $\psi$ with $\psi$, or possibly with $\psi$ and $\hat{\psi}$, where $\hat{\psi}$ is an appropriate ${\rm Spin}(n)$ invariant complex conjugation. Indeed, the tensor product of the spin representation $S$ with itself contains the spaces of all degree differential forms on $M$
\be
S\otimes S = \Lambda^\bullet (M).
\ee
The differential forms produced via this construction are automatically $G\subset{\rm GL}(n)$-invariant and encode the G-structure in question. Thus, as is seen from a large collection of examples, taking a sufficient number of differential forms obtained as spinor bilinears constructed from $\psi$ and $\hat{\psi}$ is sufficient to reproduce both the metric on $M$, and the spinor $\psi$ (the latter always mod $\Z_2$ sign ambiguity). 

Some of the known examples of such an encoding are as follows.
\begin{itemize}
\item {\bf $\{1\}$-structures in 3D.} In this case ${\rm Spin}(3)={\rm SU}(2)$, and the spinor is a 2-component spinor $S\sim \C^2$. Taking a unit such spinor $\langle \hat{\psi}\psi\rangle =1$, the stabiliser $G'=\{1\}$ is trivial. We can construct a real one-form $e^3\in \Lambda^1$ via $e^3(X) := \langle \hat{\psi}, X \psi\rangle$, and a complex-valued one-form $e(X): =\langle \psi, X \psi\rangle$, where a vector $X\in TM$ acts on $\psi$ by the Clifford multiplication, we get an orthonormal co-frame $e^{1,2,3}$, where $e=e^1+\im e^2$. One then recovers the metric via $ds^2 = (e^1)^2 + (e^2)^2 + (e^3)^2$. 
\item {\bf ${\rm SU}(2)$-structures in 4D.} In this case ${\rm Spin}(4)={\rm SU}(2)\times {\rm SU}(2)$. Taking a unit $\langle \hat{\psi},\psi\rangle=1$ spinor $\psi\in S_+$, its stabiliser is the other ${\rm SU}(2)$ that does not act on it. The construction of all possible differential forms returns one real $\omega$ and one complex 2-form $\Omega$ via
\begin{equation}\label{omega-Omega-psi}
\omega(X,Y)=\langle \hat{\psi}, X Y \psi\rangle, \qquad \Omega(X,Y) = \langle \psi, X Y \psi\rangle. 
\end{equation}
These objects satisfy 
\begin{equation}\label{omega-Omega}
\Omega\wedge \Omega=0, \qquad \Omega\wedge \omega=0, \qquad \frac{1}{2} \Omega\wedge \bar{\Omega} = \omega^2. 
\end{equation}
Alternatively, decomposing $\Omega := \Sigma^1+\im \Sigma^2$ and renaming $\omega:=\Sigma^3$ we get a triple of 2-forms satisfying $\Sigma^i \Sigma^j \sim \delta^{ij}, i,j=1,2,3$, where $\delta^{ij}$ is the standard Kronecker delta. Taking such a triple of 2-forms as the basic geometric data, one can recover the Riemannian metric $g_\Sigma$ (together with an orientation of $M$) via 
\be
g_\Sigma(X,Y) {\rm vol}_\Sigma = -\frac{1}{6} \epsilon^{ijk} i_X \Sigma^i \wedge i_Y \Sigma^j \wedge \Sigma^k.
\ee
Here ${\rm vol}_\Sigma$ is the volume form of the metric $g_\Sigma$, and the sign is to agree with our later conventions. The object $\epsilon^{ijk}$ is the standard completely anti-symmetric object with $\epsilon^{123}=+1$. 
\item {\bf ${\rm SU}(3)$-structures in 6D.} We now have ${\rm Spin}(6)={\rm SU}(4)$, and $S_+=\C^4$. The stabiliser of a semi-spinor $\psi\in S_+$ is ${\rm SU}(3)$. Taking a unit spinor $\langle \hat{\psi},\psi\rangle=1$, the construction of differential forms returns a real 2-form $\omega(X,Y)=\langle \hat{\psi}, X Y \psi\rangle$, and a complex 3-form $\Omega(X,Y,Z) = \langle \psi, X Y Z \psi\rangle$. This 3-form is decomposable, and also satisfies $\Omega\omega=0$, as well as $\Omega\wedge \bar{\Omega} = (4\im/3) \omega^3$. In this case it is sufficient to take as the basic geometric data the non-degenerate 2-form $\omega$, as well as a real 3-form ${\rm Re}(\Omega)$. The metric on $M$, as well as an almost complex structure in which $\Omega\in \Lambda^{3,0}$, are then recovered from these data by a procedure explained in e.g. \cite{Vezzoni}.
\item {\bf G${}_2$-structures in 7D.} In this case there are real spinors $\hat{\psi}=\psi$. Taking a real unit spinor $\psi\in S, \langle\psi,\psi\rangle=1$, the stabiliser ${\rm Stab}_\psi={\rm G}_2$. The construction of possible differential forms returns the 3-form $C(X,Y,Z) = \langle \psi, XYZ\psi\rangle$. There is also a 4-form, and a 7-form produced, but these are not independent and are determined by $C$. The 3-form $C$ produced by this construction is non-degenerate in a suitable sense. Moreover, the ${\rm GL}(7)$ orbit of 3-forms of this type is open in $\Lambda^3$. Taking a non-degenerate (and positive, in the sense that the quadratic form $g_C$ is definite) $C\in \Lambda^3$ as the basic geometric data, the metric (and orientation) of $M$ is recovered via
\be
g_C(X,Y) {\rm vol}_C = \frac{1}{6} i_X C \wedge i_Y C \wedge C.
\ee
\item {\bf ${\rm Spin}(7)$-structures in 8D.} In this dimension we again have real semi-spinors $\hat{\psi}=\psi$. Taking a unit real semi-spinor $\psi\in S_+$, the stabiliser is ${\rm Stab}_\psi = {\rm Spin}(7)$. The construction of differential forms returns a 4-form $\Phi(X,Y,Z,W) = \langle \psi, XYZW\psi\rangle$. This is a 4-form of a special algebraic type, whose ${\rm GL}(8)$ stabiliser is ${\rm Spin}(7)$. Taking this 4-form as the basic geometric data, the metric $g_\Phi$ is recovered by the formula in \cite{Spiro-Spin7}.
\end{itemize}
Other examples can be constructed, but we will not attempt a systematic enumeration. The purpose of the above list is to illustrate a common feature: in each case, we have a spinorial G-structure, which is metric $G\subset{\rm SO}(n)$, and determines a metric plus extra geometric data. In each case the G-structure is encoded into a collection of differential forms satisfying suitable algebraic constraints, and the Riemannian metric on $M$ can be recovered from these differential forms. We have not discussed the intrinsic torsion in each example, however, it is known that in each case in the above list of examples the intrinsic torsion is completely determined by the exterior derivative of the relevant differential forms. This fact makes the spinorial G-structures particularly interesting. 

\subsection{$\tilde{G}$-structures via $G\subset\tilde{G}$-structures.} The guiding idea of this paper is the following. In many cases, a spinorial 
G-structure can be encoded by differential forms, while a larger structure $\tilde{G}\supset G$ cannot. However, when the quotient  $\tilde{G}/G$
carries additional structure (for example, when it is itself a Lie group), one can describe $\tilde{G}$-geometry using G-structures whose defining forms are allowed to take values in an associated bundle. In this way, the geometry is encoded into differential forms, while the additional degrees of freedom present in $G$-structure as compared to the $\tilde{G}$-structure are absorbed into a gauge symmetry. 

The simplest example illustrating this way of thinking is the standard frame formalism for Riemannian geometry. Given a Riemannian metric on an oriented manifold $M$, choosing an oriented orthonormal (co-)frame for $T^*M$ "breaks" the structure group ${\rm SO}(n)$ completely. Thus, the collection of 1-forms comprising a frame can be referred to as $G=\{1\}$-structure. Nevertheless, we can describe $\tilde{G}={\rm SO}(n)$-structures this way, by allowing the frame to be a bundle-valued object, namely interpreting the frame as an object in $\Omega^1(M,\R^n)$, i.e. a 1-form with values in $\R^n$, or equivalently as a map $e: \R^n\to T^*_p M$. This becomes possible because $\tilde{G}/G={\rm SO}(n)$ a Lie group. The intrinsic torsion of this $\{1\}$-structure is then encoded in an ${\rm SO}(n)$-connection. This is of course well-known, but it is useful to spell out the relevant formulas to contrast with what happens in different setups below. We use the abstract index notation with indices $I,J,\ldots$ to denote $\R^n$-valued objects. We also use Greek letters as indices to refer to objects values in $T^*M$ and $TM$. We then have
\begin{proposition}\label{prop:frame}
There exists a unique ${\rm SO}(n)$ connection $w^I{}_J$, locally an object in $\Omega^1(M,\mathfrak{so}(n))$, which has the property that the total covariant derivative of the frame $e^I\in \Omega^1(M,\R^n)$ with respect to the Levi-Civita connection $\nabla$ (for the metric for which $e^I$ is the (co-)frame) and $w^I{}_J$ vanishes
\begin{equation}\label{nabla-frame}
\nabla_\mu e^I_\nu + w_\mu^I{}_J e^J_\nu=0.
\end{equation}
The connection $w$ is completely determined by the projection of this equation to $\Lambda^2$, i.e. 
\[
de^I + w^I{}_J\wedge e^J=0,
\]
and is thus completely determined by the exterior derivatives of the 1-forms $e^I$. 
\end{proposition}
By taking another Levi-Civita covariant derivative of the equation (\ref{nabla-frame}) and anti-symmetrising, one obtains a relation between the Riemann curvature and the curvature of the connection $w^I{}_J$
\begin{equation}
R_{\mu\nu\rho}{}^\sigma e^I_\sigma = - F_{\mu\nu}^I{}_J e^J_\rho, \qquad F_{\mu\nu}^I{}_J = 2 \partial_{[\mu} w_{\nu]}^I{}_J + 2 w^I_{[\mu}{}_K w^K{}_{\nu]}{}_J,
\end{equation}
which makes the frame formalism an extremely useful description of Riemannian geometry. 

Another example where this mechanism is at play is well-known in the context of K\"ahler geometry $\tilde{G}={\rm U}(n)$, as we now remind. In this case $G={\rm SU}(n)$, and the relevant differential forms are $\omega\in \Lambda^2$ and a complex decomposable n-form $\Omega\in \Lambda^n_\C$. The almost-symplectic 2-form $\omega$ is globally-defined, while $\Omega$ is only a section of an $H={\rm U}(1)$ complex line bundle. We then have
\begin{proposition} The intrinsic torsion of the $\tilde{G}={\rm U}(n)$-structure vanishes if and only if there exists a ${\rm U}(1)$ connection $a$ such that
\begin{equation}
\nabla_\mu \omega_{\rho\sigma} =0, \qquad \nabla_\mu \Omega_{\rho\sigma} + \im a_\mu  \Omega_{\rho\sigma} =0.
\end{equation}
Here $\nabla$ is the covariant derivative for the metric determined by the ${\rm U}(n)$ structure. Moreover, the ${\rm U}(1)$ connection $a$ is completely determined by the projection of these equations to $\Lambda^3$
\[
d\omega =0, \qquad d\Omega + \im a\wedge \Omega=0.
\]
\end{proposition}
This gives a useful and powerful description of K\"ahler geometry in terms of differential forms $\omega, \Omega$, where the second one is bundle-valued. From the point of view of $G={\rm SU}(n)$ structure, the connection $a$ is a part of its intrinsic torsion. However, it does not need to vanish for the $\tilde{G}={\rm U}(n)$ structure, and becomes an important geometric object of K\"ahler geometry - the Chern connection. This provides another model example of our general mechanism: a $\tilde{G}={\rm U}(n)$-structure can be described using an $G={\rm SU}(n)$-structure whose defining forms are allowed to be bundle-valued.

\subsection{${\rm SO}(4)$-structures via ${\rm SU}(2)$-structures.} The main purpose of this paper is to develop a similar perspective on Riemannian geometry in four dimensions, with $\tilde{G}={\rm SO}(4)$ and $G={\rm SU}(2)$. In this case $H=\tilde{G}/G={\rm SO}(3)$. We will show how Riemannian geometry can be encoded into a triple of 2-forms $\Sigma^{1,2,3}$, where the forms are not globally-defined but rather bundle-valued. Similarly, the intrinsic torsion of the $G={\rm SU}(2)$-structure becomes an ${\rm SO}(3)$ connection. One gets a very powerful formalism that allows to describe four-dimensional Riemannian geometry in terms of differential forms. What arises is precisely the Plebanski formalism for General Relativity \cite{Plebanski:1977zz}, which we thus describe from the point of view of G-structures. 

We should warn the reared that, in what follows, we will often use the term ${\rm SU}(2)$-structure in an extended sense of allowing objects to become bundle-valued. Thus, rather than requiring a reduction of the frame bundle, we take as basic data same as that of an ${\rm SU}(2)$-structure, namely a triple of 2-forms (or, equivalently, an $\R^3$-valued 2-form) satisfying certain algebraic conditions. However, these forms need not be globally defined as ordinary differential forms, but may take values in an associated ${\rm SO}(3)$ vector bundle. This allows us to describe general Riemannian metrics without imposing topological restrictions. It is only when this bundle is trivial and admits a global section that one recovers the standard ${\rm SU}(2)$-structures. Such bundle-valued ${\rm SU}(2)$-structures can be referred to as ${\rm SO}(3)$-covariant ${\rm SU}(2)$-structures, but we will sometimes say simply ${\rm SU}(2)$-structures. 

We now describe our main results more concretely. We first need a convenient for our purposes definition of an ${\rm SU}(2)$-structure. We adopt a similar definition to that in \cite{Fowdar}, see Definition 2.2, but with the opposite sign in the orientation condition. Let $M$ be an orientable smooth 4-dimensional manifold. Let $\delta^{ij}$ be the standard Kronecker delta $\delta^{ij}=1, i=j$ and $\delta^{ij}=0, i\not=j$. 
\begin{definition}  \label{def:su2} {\bf (Local)}
An (oriented) ${\rm SU}(2)$-structure on $M$ is a triple $\Sigma^{1,2,3}\in \Omega^2(M)$ of 2-forms on $M$ such that 
\begin{enumerate} 
\item (Orthonormality condition) 
\[ \Sigma^i\wedge \Sigma^j = 2 \delta^{ij} \vol_\Sigma,
\]
for some nowhere vanishing 4-form $\vol_\Sigma$. 
\item (Orientation) For any $X,Y: i_X \Sigma^1=i_Y \Sigma^2$ we have $\Sigma^3(X,Y)\leq 0$. 
\end{enumerate}
\end{definition}
According to this definition, an ${\rm SU}(2)$-structure on $M$ defines a positive-definite (with respect to the conformal metric on $\Lambda^2 T^*M$) rank 3 subspace $\Lambda^+\subset \Lambda^2$. As is well-known, this is equivalent to selecting a Riemannian signature conformal metric on $M$. The objects $\Sigma^i$ then form an orthonormal basis of $\Lambda^+$. Moreover, they define the 4-form $\vol_\Sigma$, thus fixing the full metric on $M$ together with an orientation. As we will show in the main text, this metric $g_\Sigma$ is given by the formula (\ref{urbantke}), and the 2-forms $\Sigma^i$ are self-dual with respect to $g_\Sigma$, in the orientation $\vol_\Sigma$. The second orientation condition is one that is needed to select the particular sign on the right-hand side of the formula (\ref{urbantke}), see the main text for a proof. 

The above definition gives a local description of ${\rm SU}(2)$-structures. The following alternative definition takes care of the global aspects.
\begin{definition}{\bf (Global)}
An oriented $\mathrm{SU}(2)$-structure on an oriented smooth four-manifold $M$
consists of an oriented Euclidean rank-three vector bundle $E\to M$ and an
$E$-valued two-form
\[
\Sigma\in\Omega^2(M,E)
\]
such that the following conditions are satisfied:
\begin{enumerate}
\item The associated bundle map
\[
\Sigma:E\longrightarrow \Lambda^2T^*M
\]
is an isometric (with respect to the fibre metric on $E$ and the wedge-product pairing on $\Lambda^2T^*M$) bundle embedding whose image is a rank-three subbundle
\[
\Lambda^+_\Sigma:=\Sigma(E)\subset\Lambda^2T^*M.
\]
\item For any local oriented orthonormal
frame \(e_i\) of \(E\), writing \(\Sigma=\Sigma^i e_i\), the triple
\((\Sigma^1,\Sigma^2,\Sigma^3)\) satisfies the orientation condition of
Definition~\ref{def:su2}.
\end{enumerate}
\end{definition}

As we will prove in the main text, the data of an ${\rm SU}(2)$-structure defines a Riemannian metric $g_\Sigma$ and an orientation $\vol_\Sigma$ given by 
\begin{equation}\label{urbantke}
g_\Sigma(X,Y)\,\mathrm{vol}_\Sigma
=
-\tfrac{1}{6}\,\epsilon_{ijk}\,
i_X\Sigma^i \wedge i_Y\Sigma^j \wedge \Sigma^k, \qquad \vol_\Sigma := \tfrac{1}{6}\,\delta_{ij}\,\Sigma^i \wedge \Sigma^j.
\end{equation}
We will show that the metric $g_\Sigma$ so defined is such that $\Lambda^+_\Sigma$ coincides with the bundle
$\Lambda^+_{g_\Sigma}$ of self-dual 2-forms for the metric $g_\Sigma$, in the orientation $\vol_\Sigma$. Note that this shows that the bundle 
$E$ has the same topological type as the bundle of self-dual 2-forms of a Riemannian metric on $M$.

Our first main result describes the intrinsic torsion of an ${\rm SU}(2)$-structure as an ${\rm SO}(3)$ connection.
 \begin{theorem} \label{thm:nabla-sigma}
 The torsion of the Levi-Civita connection (for the metric defined by $\Sigma$'s) of an ${\rm SU}(2)$-structure is an ${\rm SO}(3)$ connection $A^i$ such that
\be\label{intr-torsion}
\nabla_\mu \Sigma^i_{\rho\sigma}+ \epsilon^{ijk} A^j_\mu \Sigma^k_{\rho\sigma} =0.
\ee
Moreover, the ${\rm SO}(3)$ connection $A^i$ is completely determined by the projection of this equation to $\Lambda^3$
\begin{equation}\label{sigma-A}
d\Sigma^i + \epsilon^{ijk} A^j\wedge \Sigma^k=0.
\end{equation}
This is a set of algebraic equations for the components of the 1-forms $A^i$, which has a unique solution (described in the main text), thus determining $A^i$ in terms of the $d\Sigma^i$. 
\end{theorem}
Note that this statement is completely analogous to that in Proposition \ref{prop:frame}.  This result is not new, and also appeared in \cite{Fowdar}. However, here it is stated in terms of an ${\rm SO}(3)$ connection rather than a triple of 1-forms as in \cite{Fowdar}. The formula (\ref{intr-torsion}) is also standard in the theory of quaternionic K\"ahler manifolds, see formula (\ref{alpha-connection}) below. The proof we give in the main text is based on ${\rm SO}(4)$ representation theory, and appears to be new. 

Equipped with this statement, we can obtain a characterisation of the (parts of) the Riemann curvature of the metric defined by $\Sigma^i$ in terms of the curvature of $A^i$. Indeed, taking another Levi-Civita covariant derivative of (\ref{intr-torsion}) and anti-symmetrising, we get 
\[
2 R_{\mu\nu[\rho}{}^\alpha \Sigma^i_{|\alpha|\sigma]} + \epsilon^{ijk} F^j_{\mu\nu} \Sigma^k_{\rho\sigma} =0, \qquad F^i_{\mu\nu} := 2\partial_{[\mu} A^i_{\nu]} +\epsilon^{ijk} A^j_\mu A^k_\nu.
\]
Massaging this relation gives
\[
F^i_{\mu\nu} = \frac{1}{2} R_{\mu\nu}{}^{\rho\sigma} \Sigma^i_{\rho\sigma}.
\]
Recalling the decomposition of the Riemann curvature into its scalar, Ricci and Weyl parts, this immediately implies
\begin{proposition} \label{prop:curv}
The curvature of the ${\rm SO}(3)$ connection $A^i$ encodes the self-dual part of the Riemann curvature (with respect to a pair of indices). The self- and anti-self-dual parts of the 2-forms $F^i$ encode the scalar and self-dual Weyl, and tracefree part of Ricci respectively. We have
\begin{equation}
F^i_{\mu\nu} = \left( \frac{s}{12}\delta^{ij} + W_+^{ij}\right) \Sigma^j + \tilde{R}_{[\mu}{}^\alpha \Sigma^i_{|\alpha|\nu]},
\end{equation}
where $R_{\mu\nu} = R_{\mu\alpha\nu}{}^\alpha$ is the Ricci curvature, $s = g^{\mu\nu} R_{\mu\nu}$ is the scalar curvature,
\[
\tilde{R}_{\mu\nu} = R_{\mu\nu} - \frac{s}{4} g_{\mu\nu}
\]
is the tracefree part of Ricci, and
\[
\qquad \frac{s}{12}\delta^{ij} + W_+^{ij}  = \frac{1}{8} R^{\mu\nu\rho\sigma} \Sigma^i_{\mu\nu} \Sigma^j_{\rho\sigma},
\]
where $W_+^{ij}$ is the $3\times 3$ tracefree matrix encoding the self-dual part of the Weyl curvature. 
\end{proposition}
This shows how Riemannian geometry can be done via 2-forms $\Sigma^i$, and in particular shows that the Einstein condition for the metric defined by $\Sigma$'s can be imposed by requiring that $F^i\subset \Lambda^+$. This characterisation is well-known in the context of quaternion K\"ahler geometry in four dimensions, see below, so this proposition is not new. 

We also point out that the proposed point of view on Riemannian geometry in four dimensions has the advantage that both K\"ahler and hyper-K\"ahler geometries are automatically encompassed. Indeed, they both correspond to additional reductions of the associated ${\rm SO}(3)$ bundle in which $\Sigma^i$ are valued. K\"ahler geometry corresponds to a ${\rm U}(1)$ reduction by a preferred 2-form $\omega$, where ${\rm U}(1)$ then acts by rotations in the plane orthogonal to $\omega$. Hyper-K\"ahler geometry corresponds to a $\{1\}$ reduction, in which the bundle of $\Sigma$'s is trivial and admits a global frame. 

We now describe our main new result. One of the most interesting consequences of the developed viewpoint is that $\Sigma$'s whose associated metric is Einstein arise as critical points of a natural and {\bf unique} action functional. To describe this, we note that infinitesimal diffeomorphisms act on $\Sigma^i$ by the Lie derivative. This action, together with the infinitesimal action of ${\rm SO}(3)$ on $\Sigma^i$, are described by the following formulas
\be\label{intro-transforms}
\delta_X \Sigma^i = d i_X \Sigma^i + i_X d\Sigma^i , \qquad \delta_\phi \Sigma^i = [\phi, \Sigma]^i.
\ee
Here $X\in \Gamma(TM)$ is a vector field, and $\phi\in \Gamma(E)$ is a section of an $\R^3$ vector bundle over $M$. 
\begin{theorem}\label{thm:action}
There is a unique (up to an overall multiple) action functional $S[\Sigma]$ of the type (\ref{gen-action-intr}) that is diffeomorphism and ${\rm SO}(3)$ invariant, i.e. invariant under both transformations in (\ref{intro-transforms}). It is given by
\be\label{intr-action}
S[\Sigma] = -\frac{1}{2} \int \epsilon^{ijk} \Sigma^i \wedge A^j \wedge A^k.
\ee
Here $A$ is the canonical ${\rm SO}(3)$ connection (intrinsic torsion) determined by $\Sigma$'s via (\ref{sigma-A}). The critical points of this action are $\Sigma$'s whose associated metric is Einstein. 
\end{theorem}
The coefficient in front of the action is chosen for future convenience and will be explained in the main text. There is also a first order in derivatives version of (\ref{intr-action}), which is a functional of both $\Sigma$ and an additional set of $\R^3$-valued 1-form fields, which, after one imposes their Euler-Lagrange equations, become identified with $A^i$. As we shall see in the main text, this first order version of the action principle is what is known in the physics literature as the Plebanski formalism for General Relativity, see \cite{Plebanski:1977zz}. 

The mechanism at play in the above construction is that $\tilde{G}/G=H$ is a Lie group, and this allowed to describe $\tilde{G}$-structures as $H$-covariant $G$-structures. In particular, this leads to the preferred action functionals of the general type (\ref{gen-action-intr}) that are locally $H$-invariant. It would be interesting to determine such functionals for the example $\tilde{G}={\rm U}(n), G={\rm SU}(n)$, i.e. almost-Hermitian structures, but this is beyond the scope of this paper. For a general $G$-structure no such action functional selection principle is possible. However, there could still be functionals of the type (\ref{gen-action-intr}) that possess (local) symmetry that is larger than that of diffeomoprhisms. It would be very interesting to find examples. Another interesting and important question is to analyse the geometric flows that arise as the gradient flows of the action functionals of the type (\ref{gen-action-intr}). For ${\rm SU}(2)$-structures this has been studied in \cite{Fowdar}, but not for the action functional (\ref{intr-action}). Some information about its gradient flow can be easily extracted from the results in the main text, but we leave the complete treatment to a separate work. 

\subsection{Almost quaternion Hermitian manifolds.} We end this Introduction with comments on parallels between our story and the well-known theory of almost quaternion Hermitian manifolds. The standard reference for the latter is \cite{BG}, Chapter 12.

An almost quaternionic structure on $M^{4n}$ is a reduction of the frame bundle $FM$ to an ${\rm GL}(n,\mathbb{H}) {\rm Sp}(1)$ subbundle. Such a reduction is achieved by choosing a rank 3 subbundle $Q\subset {\rm End}(TM)$ of the endomorphism bundle, with a local section $(J^1, J^2, J^3)$ satisfying the algebra of imaginary quaternions. An almost quaternionic Hermitian manifold $M^{4n}$ is a further reduction to ${\rm Sp}(n) {\rm Sp}(1)$, where the endomorphisms $(J^1, J^2, J^3)$ are required to be compatible with a Riemannian metric on $M^{4n}$. 
There are subtle obstructions to existence of ${\rm Sp}(n) {\rm Sp}(1)$ structures for $n>1$, while for $n=1$ we have ${\rm Sp}(1) {\rm Sp}(1) = {\rm SO}(4)$, so any orientable Riemannian manifold has this structure. It is this case that has many intersections with our formalism. 

An almost quaternionic Hermitian structure is called quaternionic K\"ahler if the Levi-Civita connection preserves the quaternionic structure. Given an almost quaternionic Hermitian structure, one can equivalently consider the triple of 2-forms $(\Sigma^1, \Sigma^2,\Sigma^3)$. These define the canonical 4-form
\be
\Omega = \Sigma^1\wedge \Sigma^1 +\Sigma^2\wedge \Sigma^2 +\Sigma^3\wedge \Sigma^3.
\ee
One then shows that, for $n>1$ the manifold is quaternionic K\"ahler if and only if $\nabla \Omega=0$. For $n>2$ one can show that an almost quaternionic Hermitian manifod is quaternionic K\"ahler if $d\Omega=0$. In the case of 8D manifolds this is not sufficient, one also needs the condition that the subbundle $Q^*\subset \Lambda^2 TM$ is a differential ideal. The case $n=1$ is special, see below.

For any quaternionic K\"ahler manifolds the subbundle $Q^*\subset \Lambda^2 TM$ is preserved by the Levi-Civita connection. Because of this, we have
\be\label{alpha-connection}
\nabla \Sigma^i =- A^{ij} \wedge \Sigma^j,
\ee
where $A^{ij} = - A^{ji}$ is a set of 1-forms, and the minus on the right-hand side of this formula is for agreement with our previous formulas. These form an ${\rm SO}(3)$ connection $A$ whose curvature $F=dA + A\wedge A$ contains information about a part of the Riemann curvature. This curvature can be viewed as a map
\be
F: \Lambda^2 TM \to {\rm SkewEnd}(Q^*).
\ee
For any quaternionic K\"ahler manifold one can show that $F= \lambda \pi_{Q^*}$, where $\pi_{Q^*}$ is the pointwise orthogonal projection $\pi_{Q^*} : \Lambda^2 TM \to Q^*$. Here $\lambda$ is a multiple of scalar curvature. 

The case $n=1$ is special. In this case $\Omega$ coincides with the volume form, so the condition $\nabla \Omega=0$ is automatically satisfied. In this case $Q^*$ coincides with the subbundle $Q=\Lambda^+$ of self-dual 2-forms. As such, it is automatically preserved by the Levi-Civita connection. We then have the connection 1-forms $A^{ij}$, as well as the curvature $F$. One then defines an almost quaternionic Hermitian 4-manifold to be quaternion K\"ahler if $F= \lambda \pi_{Q^*}$. It is then not hard to show, and follows from our Proposition \ref{prop:curv}, that a 4-manifold is quaternion K\"ahler if and only if it is Einstein and one of the two halves $W^\pm$ of its Weyl curvature vanishes. 

The difference between our formalism and the setup of almost quaternionic Hermitian 4-manifolds is in the starting point. In the latter case one starts with a Riemannian metric, plus a triple of almost complex structures satisfying the algebra of imaginary quaternions. One then passes to the corresponding 2-forms $\Sigma^i, i=1,2,3$, which are self-dual. Because of this one can write the equation (\ref{alpha-connection}), which is central for the whole theory, and gives access to (a part of) the Riemann curvature. In our case, crucially, we start with a bundle-valued triple of 2-forms $\Sigma^i$, and the metric is defined by this data rather than being independently specified. In this approach it takes some work to arrive at the equation (\ref{alpha-connection}), as we explain in the main text. 

\subsection{Outline of the paper.} We start by describing ${\rm SU}(2)$ structures in more detail, and show how an ${\rm SU}(2)$ structure on $M$ defines a metric. There is some new material in this section, in particular we give a new formula (\ref{new-metric}) for the metric in terms of an ${\rm SU}(2)$ structure. We then proceed to a description of how spaces of $E$-valued 1- and 2-forms on $M$ split into their irreducible components with respect to the action of ${\rm SU}(2)\times{\rm SU}(2)$. Some of the material here is new, in particular the decomposition of $E\otimes \Lambda^2$ using the operator $J_2$ defined in (\ref{J2}). We then discuss, in Section \ref{sec:torsion}, the notion of the intrinsic torsion of an ${\rm SU}(2)$ structure. We also characterise how (a part of the) Riemann curvature is determined by the intrinsic torsion, and how the Einstein condition can be imposed in this language. We analyse and construct diffeomorphism and ${\rm SO}(3)$-invariant functionals that are second order in derivatives and quadratic in perturbations of an ${\rm SU}(2)$ structure in Section \ref{sec:linearised}. The key result of this section is that there is a unique diffeomorphism and ${\rm SO}(3)$-invariant action functional. We construct non-linear action functionals in Section \ref{sec:actions}. This section contains the proof of Theorem \ref{thm:action} described in the Introduction.

\section{${\rm SU}(2)$-structures in four dimensions}
\label{sec:structures}

The geometry of ${\rm SU}(n)$-structures was studied in \cite{Cabrera}, \cite{Cabrera-Swann}. In particular, it was shown that the intrinsic torsion of the ${\rm SU}(n)$-structure is determined by the exterior derivatives of the differential forms encoding the structure. Another closely related to our goals reference is \cite{Fowdar}, where ${\rm SU}(2)$-structures in four dimensions were studied, intrinsic torsion analysed, and some second order in derivatives functionals constructed. Other related papers are \cite{Cabrera-Harmonic}, \cite{Fadel}. However, the treatment in this paper differs from all these references in the key fact that we use  what can be called ${\rm SO}(3)$-covariant ${\rm SU}(2)$-structures to describe ${\rm SO}(4)$-structures. The purpose of this section is to explain this concept in details, and explore its consequences. 

\subsection{2-forms and geometric structures in four dimensions.} We start by recalling, following \cite{Donaldson}, that many differential geometric structures in four dimensions can be fruitfully encoded by 2-forms. In all these examples $M$ is a smooth oriented 4-manifold.
\begin{itemize}
\item A {\bf conformal structure} on $M$ can be encoded by a 3-dimensional subbundle $\Lambda^+\subset \Lambda^2$ on which the wedge product form is strictly positive. A {\bf Riemannian metric} is specified when in addition a volume form is chosen.
\item A (compatibly oriented) {\bf symplectic structure} on $M$ is a closed 2-form $\omega$ which is positive $\omega\wedge \omega>0$. 
\item An {\bf almost-complex structure} on $M$ is given by a 2-dimensional oriented real subbundle $\Lambda^{2,0}\subset \Lambda^2$ on which the wedge product form is strictly positive. Oriented real 2-dimensional bundles are equivalent to complex line bundles, so we can alternatively view $\Lambda^{2,0}$ as this complex line bundle. 
\item A {\bf complex structure} is given by an almost-complex structure $\Lambda^{2,0}$ such in the neighbourhood of each point there exists a pair of closed 2-forms $\Sigma^{1,2}$ such that $\Sigma^1\wedge\Sigma^1 = \Sigma^2\wedge \Sigma^2$ and $\Sigma^1\wedge \Sigma^2=0$ which span $\Lambda^{2,0}$. Alternatively, there exists a closed complex 2-form $\Omega:=\Sigma^1+\im \Sigma^2$ that is decomposable $\Omega\wedge \Omega=0$ and spans $\Lambda^{2,0}$ viewed as a complex line bundle.
\item A {\bf K\"ahler structure} is a complex structure together with a symplectic structure $\omega$ that is wedge product orthogonal to $\Lambda^{2,0}$.  
\item A {\bf hyper-K\"ahler structure} is a triple of closed 2-forms $\Sigma^{1,2,3}$ that satisfy $\Sigma^i\wedge \Sigma^j\sim \delta^{ij}$, or alternatively $\Sigma^1\wedge \Sigma^1=\Sigma^2\wedge \Sigma^2=\Sigma^3\wedge \Sigma^3$ and $\Sigma^i\wedge \Sigma^j = 0, i\not=j$.
\end{itemize}
In view of these examples, a very natural question is how to do geometry when the relevant geometric structure is encoded by 2-forms. While the answer to this question is widely known in the case of K\"ahler structures, as was reviewed in the Introduction, the formalism that allows to deal with Riemannian metrics is not well-known, and one of the main aims of this paper is to develop and explain it. 

\subsection{${\rm SU}(2)$-structures from spinors.} We now recall in more details how ${\rm SU}(2)$-structures in four dimensions, together with triples of 2-forms $\Sigma^{1,2,3}$, arise from spinors. We first develop the scenario where the spinor $\psi$ is globally-defined, so that the constructed from it 2-forms are globally-defined. Eventually, in our application to the case of Riemannian geometry, we will need to relax the assumption that $\psi$ is globally-defined. 

Let $M$ be an oriented Riemannian 4-manifold, which we for now assume to be spin, so that a spin structure exists, and for simplicity simply connected, so that this spin structure is unique. The double cover of ${\rm SO}(4)$ is ${\rm Spin}(4)={\rm SU}(2)\times{\rm SU}(2)$. As is standard, the best way to see this is to identify $\R^4=\mathbb{H}$, and ${\rm SU}(2)={\rm Sp}(1)$, the group of unit quaternions. Then ${\rm Sp}(1)\times{\rm Sp}(1)$ acts on $x\in \mathbb{H}$ via $(q_1,q_2): x\to q_1 x q_2^{-1}$. This gives the double-covering map $\lambda: {\rm SU}(2)\times{\rm SU}(2)\to {\rm SO}(4)$, with $(-1,-1)\in {\rm Sp}(1)\times{\rm Sp}(1)$ sent to the identity in ${\rm SO}(4)$. We now take a (unit $\langle \hat{\psi},\psi\rangle=1$) semi-spinor $\psi\in S_+$, on which ${\rm Spin}(4)$ acts via $(q_1,q_2): \psi \to q_1 \psi$. Any such spinor is stabilised by the copy of ${\rm SU}(2)$ that does not act on it. We then construct the 2-forms $\omega,\Omega$ from $\psi$, via (\ref{omega-Omega-psi}), which satisfy (\ref{omega-Omega}). Alternatively, changing the names to $\omega=\Sigma^3, \Omega=\Sigma^1+\im \Sigma^2$ the real 2-forms $\Sigma^{1,2,3}$ satisfy $\Sigma^i\wedge \Sigma^j\sim \delta^{ij}$. These 2-forms are invariant under the action of $\lambda({\rm SU}(2))={\rm SU}(2)\subset {\rm GL}(4)$ (this will be verified below), and, when globally-defined, provide a reduction of the principal ${\rm GL}(4)$ frame bundle to an ${\rm SU}(2)$ subbundle, and thus an ${\rm SU}(2)$-structure. 

\subsection{Hyper-K\"ahler and K\"ahler geometries.} Let us now discuss the natural differential equations that can be imposed on such an ${\rm SU}(2)$-structure. When all 3 of these 2-forms are closed $d\Sigma^i=0$, this ${\rm SU}(2)$-structure is integrable, and we obtain a hyper-K\"ahler structure. This statement is well-known, but will also follow from what is proven below for ${\rm SO}(3)$ covariant ${\rm SU}(2)$-structures. In this case the Riemannian holonomy is reduced to (a subgroup of) ${\rm SU}(2)$, the associated metric is Ricci flat. In this case the spinor $\psi$ that we started from is parallel $\nabla\psi=0$ with respect to the (spinor lift of) the Levi-Civita connection. 

Instead, we can impose a weaker condition
\begin{equation}\label{d-omegas}
d\omega = 0, \qquad d\Omega = \im a\wedge \Omega,
\end{equation}
for some $a\in \Lambda^1$. In this case, the almost-complex structure defined by $\Omega$ is still integrable, and we have a K\"ahler structure. Note that in this case we only need $\omega$ to be globally-defined, while $\Omega$ can be defined modulo multiplication by a phase $\Omega\to e^{\im\phi(x)} \Omega$, where $\phi(x)$ is a function. The second equation in (\ref{d-omegas}) continues to make sense even when $\Omega$ is only defined up to a phase, because if it is satisfied with one choice of $\Omega$, it will be satisfied with a different choice of $\Omega$ and a different, ${\rm U}(1)$ gauge-transformed connection $d(e^{\im \phi}\Omega) = \im (a+ d\phi) \wedge \Omega$. In other words, in this case $\Omega$ is a section of a complex line bundle, and $a$ is a ${\rm U}(1)$ connection on this line bundle. Also, in this case the spinor $\psi$ does not need to be globally-defined, for it is sufficient for it to be defined only up to a phase. In particular, this means that $M$ does not even need to be spin. 

An alternative way of describing the K\"ahler setting is to say that it arises from a unit semi-spinor, but the spinor $\psi$ is only required to be fixed up to scale $q_1 \psi\sim \psi$. The (projective) stabiliser of $\psi$ is now larger and is given by $\tilde{G}={\rm U}(1)\times {\rm SU}(2)={\rm U}(2)$. So, this gives a $\lambda({\rm U}(2))={\rm U}(2)$ structure. In this case the spinor defines the real 2-form $\omega$, and 2-form $\Omega$ only modulo phase. This is sufficient to define the metric together with an almost-complex structure. As we already commented on, the equations (\ref{d-omegas}) are still meaningful even though $\Omega$ is only defined up to a phase. 

In the first hyper-K\"ahler example, the 2-forms $\Sigma^{1,2,3}$ are globally well-defined, which implies topological restrictions on $M$. First, because $\Sigma^{1,2,3}$ arise from a spinor construction, the spinor $\psi$ needs to be globally-defined, and so $M$ must be spin. Further, $\Sigma^{1,2,3}$ are all self-dual in the metric that they define via (\ref{urbantke}), and thus give an orthonormal basis of $\Lambda^+$. The existence of a global such basis implies that $\Lambda^+$ is a trivial vector bundle. It is well-known that this implies 
\[
2\chi(M) + 3\tau(M)=0,
\]
where $\chi(M)$ is the Euler characteristic of $M$, and $\tau(M)$ is the signature. This is very restrictive, in particular neither $\mathbb{CP}^2$ (because it is not spin) nor $S^4$ (because $2\chi(S^4) + 3\tau(S^4)=4\not=0$) satisfies these restrictions. 

In the second K\"ahler example, the unit spinor $\psi$ does not need to exist globally, but rather only up to a phase. This means that in this case $M$ does not need to be spin. For the arising 2-forms, only the 2-form $\omega$ is globally-defined and, in particular, defines an almost-complex structure. The arising topological restrictions are that there must exist an integral lift $c_1(M)\in H^2(M,\mathbb{Z})$ of the second Whitney class $w_2(M)$ such that $c_1^2=2\chi(M) + 3\tau(M)$. This is less restrictive than in the ${\rm SU}(2)$-structure case, but the 4-sphere $S^4$ still does not admit this structure. 

\subsection{${\rm SU}(2)$-structures from ${\rm SO}(4)$-structures.} To see how triples $\Sigma^{1,2,3}$ arise in Riemannian geometry in four dimensions, we start from an oriented Riemannian 4-manifold. This defines the notion of which 2-forms are self- and anti-self dual, and so defines the rank 3 subbundle $\Lambda^+\subset \Lambda^2$. We can choose a basis $\Sigma^{1,2,3}$ of $\Lambda^+$ which consists of 2-forms orthonormal $\Sigma^i\wedge \Sigma^j\sim \delta^{ij}$ with respect to the wedge product, normalised such that  $\Sigma^1\wedge \Sigma^1=\Sigma^2\wedge \Sigma^2=\Sigma^3\wedge \Sigma^3$ is a constant multiple of the metric volume form.  To see how the group ${\rm SO}(4)$ acts on this basis, it is most convenient to first consider the spinor lift to ${\rm Spin}(4)$, in which $\R^4=\mathbb{H}$ with the action of $(q_1,q_2)\in {\rm Spin}(4) = {\rm Sp}(1)\times{\rm Sp}(1)$ being as before $x\to q_1 x q_2^{-1}$. Then elements of $\Lambda^+$ are identified with imaginary quaternions $\Lambda^+ = {\rm Im}(\mathbb{H})$, with the action of ${\rm Spin}(4)$ being $\Lambda^+ \ni \Sigma\to q_1 \Sigma q_1^{-1}$. This gives a homomorphism ${\rm Spin}(4)\to {\rm SO}(3)$, with the kernel given by $(\pm 1, q_2)$ which factors through to a homomorphism ${\rm SO}(4)\to {\rm SO}(3)$, with kernel ${\rm SU}(2)$. Thus, ${\rm SO}(4)$ acts on $\Lambda^+$ as ${\rm SO}(3)$. 
 
A related point is that the space $\Lambda^+$ has a natural orientation. Indeed, given $\Sigma^{1,2}\in \Lambda^+$, we can use the metric to raise one of the two indices of $\Sigma^{1,2}$ to convert them into endomorphisms of $\Lambda^1$. We can then compute the commutator $[\Sigma^1,\Sigma^2]$, and lower the index so that the result is again in $\Lambda^+$. We then declare $\Sigma^1,\Sigma^2,[\Sigma^1,\Sigma^2]$ to be the positive orientation. With the ${\rm SO}(4)$ acting on $\Lambda^+$ as ${\rm SO}(3)$, it is orientation-preserving. To complete the choice of the objects $\Sigma^{1,2,3}$, we require them to form an oriented orthonormal basis of $\Lambda^+$. 

An alternative and very convenient viewpoint is to view $\Sigma^i, i=1,2,3$ as a map $\Sigma: \R^3\to \Lambda^+$. Choosing the standard metric on $\R^3$ we require this map to be an isometry. This is the meaning of the formula $\Sigma^i\wedge \Sigma^j\sim \delta^{ij}$, where the object $\delta^{ij}$ is the chosen metric on $\R^3$. Choosing an orientation of $\R^3$, and thus the object $\epsilon^{ijk}$ on $\R^3$, we also require the map $\Sigma: \R^3\to \Lambda^+$ to be orientation-preserving. Thus, the objects $\Sigma^i$ form a suitably normalised (by the volume form) oriented orthonormal frame for the bundle $\Lambda^+$, with ${\rm SO}(4)$ acting on this frame by ${\rm SO}(3)$ rotations. 

\subsection{Recovering metric from an ${\rm SU}(2)$-structure.} We now reverse the above construction, and encode the Riemannian metric by the data $\Sigma^{1,2,3}$ rather than starting from the metric. To do this, we start with an oriented 4-manifold $M$, and a triple of 2-forms $\Sigma^{1,2,3}$ spanning a positive subspace (with respect to the wedge product)  $\Lambda^+\subset \Lambda^2$. We also require $\Sigma^i\wedge \Sigma^j\sim \delta^{ij}$, so that $\Sigma^i$ is an orthonormal basis of $\Lambda^+$ that they span. This already defines a conformal metric together with the volume form, and thus a full Riemannian metric. In addition, we view $\Sigma_p:\R^3\to \Lambda^2 T^*_p M$ as an isometry from $\R^3$ with a fixed metric $\delta^{ij}$ and $\Lambda^2 T^*_p M$ with its wedge product metric. We also fix an orientation of $\R^3$ and require $\Sigma_p:\R^3\to \Lambda^2 T^*_p M$ to be orientation-preserving. This leads to the definition \ref{def:su2}. As we will now show, the metric can then be recovered from the data $\Sigma^i$ by the formula (\ref{urbantke}). The objects $\Sigma^i$ become an oriented orthonormal frame for the bundle $\Lambda^+_\Sigma$ of the metric that $\Sigma$'s define. It is also clear that we can view $\Sigma^i$ is a bundle-valued 2-form, and that no assumption about $\Sigma^{1,2,3}$ being globally-defined is necessary. 

The following dimension count is useful to convince oneself that this encoding of a Riemannian metric on $M$ is possible. The dimension of the space of 2-forms in four dimensions is ${\rm dim}(\Lambda^2)=6$, and so a triple of 2-forms needs 18 numbers to be specified. There are 5 relations in $\Sigma^i\Sigma^j\sim \delta^{ij}$, and so the dimension (per point) of the space of ${\rm SU}(2)$ structures on $M$ is $18-5=13$. This is the same as the dimension of the coset space ${\rm GL}(4,\R)/{\rm SU}(2)$. The fact the ${\rm GL}(4,\R)$ stabiliser of an  ${\rm SU}(2)$ structure on $M$ is one of the two ${\rm SU}(2)$'s in ${\rm SU}(2)\times{\rm SU}(2)/\mathbb{Z}_2 = {\rm SO}(4)\subset {\rm GL}(4,\R)$ follows from the following statement.
\begin{theorem} An ${\rm SU}(2)$ structure on $M$ determines a Riemannian metric $g_\Sigma$ as well as an orientation $\vol_\Sigma$ on $M$. The orientation is defined via
\be
{\rm vol}_\Sigma = \frac{1}{6} \Sigma^i \Sigma^i.
\ee
The metric is determined according to the following formula
\be\label{metric-formula}
g_\Sigma(X,Y) {\rm vol}_\Sigma = - \frac{1}{6} \epsilon^{ijk} i_X \Sigma^i i_Y \Sigma^j \Sigma^k.
\ee
The 2-forms $\Sigma^i$ are self-dual with respect to the Hodge star operator corresponding to $g_\Sigma$, in the orientation ${\rm vol}_\Sigma$. There exists a choice of a co-frame $e^{1,2,3,4}$ such that we have the following canonical expression for $\Sigma^i$
\be\label{Sigma-can}
\Sigma^1 = e^{41}-e^{23}, \quad \Sigma^2 = e^{42}-e^{31}, \quad \Sigma^3=e^{43}-e^{12}.
\ee
\end{theorem}

\begin{proof} Let us form $\Omega= \Sigma^1+\im \Sigma^2$. The algebraic relations satisfied by $\Sigma^i$ imply that $\Omega\Omega=0$, which then means that it is decomposable. Let us denote by $u,v$ some complex one-forms such that $\Omega = u\wedge v$. Because $\Omega\bar{\Omega}\not = 0$, it is clear that $u,v,\bar{u},\bar{v}$ span $\Lambda^1(M)$. The real 2-form $\Sigma^3$ satisfies $\Sigma^3\Omega=\Sigma^3\bar{\Omega}=0$, and is therefore of the form
\be
\Sigma^3 = \frac{1}{2\im} \alpha u\wedge\bar{u} +\frac{1}{2\im} \beta v\wedge\bar{v} + \frac{1}{2\im}\gamma u\wedge\bar{v} +\frac{1}{2\im}\bar{\gamma} v\wedge\bar{u}, \quad \alpha,\beta \in \R, \gamma \in \C.
\ee
Alternatively
\be\label{sigma3}
\Sigma^3 = \frac{1}{2\im}\left(\begin{array}{cc} u & v \end{array}\right) \left(\begin{array}{cc}  \alpha & \gamma \\ \bar{\gamma} & \beta \end{array}\right) \left( \begin{array}{c} \bar{u} \\ \bar{v}\end{array}\right).
\ee
Moreover
\be
\Sigma^3\Sigma^3 = \frac{1}{2}(\alpha\beta -  |\gamma|^2) u\wedge v\wedge \bar{u}\wedge\bar{v}.
\ee
But we must have $\Sigma^3\wedge\Sigma^3=(1/2)\Omega\wedge\bar{\Omega}$, and so $\alpha\beta -  |\gamma|^2=1$.
At the same time, the one-forms $u,v$ are defined modulo
\be\label{sl2-transform}
\left(\begin{array}{cc} u & v \end{array}\right) \to \left(\begin{array}{cc} u & v \end{array}\right) \left(\begin{array}{cc} a & b \\ c & d \end{array}\right), \qquad a,b,c,d\in \C, ad-bc=1.
\ee
The Hermitian unimodular matrix that appears in (\ref{sigma3}) can always be brought by a ${\rm SL}(2,C)$ transformation into the form $\pm \id$. This means that there is a choice of $u,v$ such that 
\be\label{canonical-form}
\Omega = \Sigma^1+\im \Sigma^2 = u\wedge v, \qquad \Sigma^3 = \pm\left(\frac{1}{2\im} u\wedge\bar{u} + \frac{1}{2\im} v\wedge\bar{v}\right),
\ee
where both signs in the expression for $\Sigma^3$ are possible. The 1-forms $u,v$ are defined modulo an ${\rm SU}(2)$ transformation, which leaves $\Sigma^3, \Omega$ invariant. This shows that the triple of 2-forms $\Sigma^i$ satisfying $\Sigma^i \Sigma^j\sim \delta^{ij}$ defines the Riemannian signature metric 
\be
g_\Sigma = |u|^2 + |v|^2,
\ee
as well as the canonical orientation of $M$ given by $\Omega\wedge\bar{\Omega} = u\wedge v\wedge \bar{u}\wedge\bar{v}$. 

We now choose a real basis of co-vectors
\be
u= e^4-\im e^3, \qquad v= e^1 +\im e^2,
\ee
so that $\Sigma^i$ take the following form
\be
\Sigma^1 = e^{41}-e^{23}, \quad \Sigma^2= e^{42}-e^{31},
\ee
while
\be
\Sigma^3= \pm ( e^{43} - e^{12}),
\ee
where the notation is $e^{I\ldots J} = e^I \wedge \ldots \wedge e^{J}$. However, we have not yet used our orientation condition on $\Sigma^{1,2,3}$ It is easy to see that it selects the plus sign in the expression for $\Sigma^3$
\be
\Sigma^3= e^{43} - e^{12}.
\ee
The metric is
\be
g_\Sigma = \sum_{I=1}^4 (e^I)^2,
\ee
and the orientation determined by $\Sigma^i$ is $e^{1234}$. It is clear that $\Sigma^i$ are self-dual 2-forms with respect to the Hodge star corresponding to the metric $g_\Sigma$, in the orientation $e^{1234}$. The formula (\ref{metric-formula}) can now be checked by a verification.
\end{proof}

We will also give an alternative explicit expression for the metric, in the spirit of the 8D formula in \cite{Spiro-Spin7}, is as follows.
\begin{theorem} The expression 
\be\label{new-metric}
 - \frac{1}{2} \frac{ (\epsilon^{ijk} i_X \Sigma^i i_X \Sigma^j i_X \Sigma^k)(e_1, e_2, e_3)}{ (i_X \Sigma^i \Sigma^i) (e_1, e_2, e_3)},
\ee
where $e_{1,2,3}$ is an arbitrary triple of vectors that together with $X$ span $TM$, is independent of the choice of $e_i$. It is homogeneity degree two in $X$, and equals the metric pairing $g_\Sigma(X,X)$. 
\end{theorem}
\begin{proof} The homogeneity degree in $X$ is obvious. To prove independence of $e^i$, let us take some other triple of vectors related to $X,e^i$ as 
\[
(e_i)' = \kappa_i X + \lambda_i{}^j e_j.
\]
From $i_X \Sigma^i i_X \Sigma^i =0$ it follows that the denominator is independent of $\kappa_i$. The numerator is independent of $\kappa_i$ because the insertion of two factors of $X$ into $\Sigma^i$ vanishes. To demonstrate independence of $\lambda_i{}^j$ we note 
\[
(\epsilon^{ijk} i_X \Sigma^i i_X \Sigma^j i_X \Sigma^k)((e^1)', (e^2)', (e^3)')= 
{\rm det}(\lambda) (\epsilon^{ijk} i_X \Sigma^i i_X \Sigma^j i_X \Sigma^k)(e^1, e^2, e^3).
\]
Similarly
\[
(i_X \Sigma^i \Sigma^j)((e^1)', (e^2)', (e^3)') = {\rm det}(\lambda) (i_X \Sigma^i \Sigma^j)(e^1, e^2, e^3).
\]
The statement of independence of choice of $e^i$ follows. The statement that this expression computes the metric pairing $g_\Sigma(X,X)$ follows by computing it for one of the frame vectors when $\Sigma^i$ is given by its canonical expression (\ref{Sigma-can}) in the co-frame basis. 
\end{proof}

The statement that the ${\rm GL}(4,\R)$ stabiliser of a triple $\Sigma^i$ satisfying $\Sigma^i\Sigma^j\sim \delta^{ij}$ is ${\rm SU}(2)$ now follows from the facts established above. First, given that $\Sigma^i$ define a Riemannian signature metric, their ${\rm GL}(4,\R)$ stabiliser is contained in ${\rm O}(4)$. Further, we have explicitly determined this stabiliser to be the ${\rm SU}(2)$ that mixes the complex one-forms $u,v$ and leaves $\Sigma^{1,2,3}$ invariant. The other ${\rm SU}(2)$ acts on $\Sigma^{1,2,3}$ as ${\rm SO}(3)$, by mixing them. 

This discussion establishes that we can think about ${\rm SO}(4)$ structures as ${\rm SO}(3)$ covariant ${\rm SU}(2)$-structures, by introducing an ${\rm SO}(3)$ vector bundle of vector-valued 2-forms, of which $\Sigma^{1,2,3}$ is a frame.

\section{Decomposition of $E$-valued differential forms}
\label{sec:decomp}

The technical tool we need to prove Theorem \ref{thm:nabla-sigma} is the decomposition of $E$-valued differential forms on $M$ into representations of the structure group in question. Such a technique is standard in the literature, see e.g. \cite{Spiro-G2} for the treatment in the $G_2$-case. As we previously explained, we view ${\rm SU}(2)$-structures so that there is no reduction of the structure group ${\rm SO}(4)$, because 2-forms $\Sigma^i$ are considered as bundle-valued objects, with ${\rm SO}(3)$ acting on them. So, the relevant representation theory is that of ${\rm SO}(4)$. This is what makes our treatment different from that in \cite{Fowdar}. While many of the formulas in this reference are ${\rm SO}(3)$ covariant, this plays no role in its considerations. In contrast, this fact will be heavily used in our constructions, greatly simplifying the resulting formulas. We will also see that, in our ${\rm SO}(3)$-covariant treatment, the intrinsic torsion decomposes into just two irreducible parts, as compared to many such parts in \cite{Fowdar}. 

\subsection{${\rm SO}(4)$ representation theory.} For any $G$-structure, the intrinsic torsion of a $G$-structure is an object that measures the failure of this $G$-structure to be integrable. From general principles, it follows that the intrinsic torsion is an object that takes values in $T^*M \otimes \mathfrak{g}^\perp$, see for example a discussion in the subsection 4.2 of \cite{Bryant-G2}. Here $\mathfrak{g}^\perp$ is understood as the associated vector bundle $\mathfrak{g}^\perp \equiv P\times_G \mathfrak{g}^\perp$, where $P$ is the principal $G$-bundle defining the $G$-structure. At the same time, the intrinsic torsion should be determinable from the covariant derivative (computed with respect to the Levi-Civita connection) of the tensors that define the $G$-structure. 

In our case, $\mathfrak{g}^\perp=\mathfrak{su}(2)=\R^3$, and so the intrinsic torsion must be an object with values in $T^*M \otimes E\equiv \Lambda^1(M)\otimes E$. The tensor that defines an ${\rm SU}(2)$-structure takes values in $\Lambda^2 T^*M \otimes E\equiv \Lambda^2(M) \otimes E$, and so does its covariant derivative $\nabla_X \Sigma$ in any direction $X\in TM$. For this reason, we need to understand the decomposition of the spaces $\Lambda^1(M) \otimes E, \Lambda^2 (M) \otimes E$, into irreducible representations of ${\rm SO}(4)={\rm SU}(2)\times{\rm SU}(2)/\Z_2$. 

Irreducible representations of ${\rm SU}(2)$ are the spin $k/2$ representations that we denote by $S^k$. They are of dimension ${\rm dim}(S^k) = k+1$. As we have previously discussed, there are two different ${\rm SU}(2)$'s in the game. One ${\rm SU}(2)$ is the group with respect to which the 2-forms $\Sigma^i$ are invariant. We will choose to denote this copy of ${\rm SU}(2)$ by ${\rm SU}_-(2)$, and the corresponding representations by $S^k_-$. The other ${\rm SU}(2)$ is one that acts on $\Sigma^i$ as in the adjoint representation. We will denote it by ${\rm SU}_+(2)$, and the corresponding representations by $S_+^k$. We then have the following standard facts
\be
\Lambda^1(M) = S_+\otimes S_-, \qquad \Lambda^2(M) = S_+^2 \oplus S_-^2, \qquad E = S_+^2.
\ee
Here we use the same notation $S_\pm$ for the spinor bundles and for their typical fibres, and interpret all representation-theoretic statements fibrewise.
This results in the following decomposition of $\Lambda^1(M)\otimes E, \Lambda^2(M)\otimes E$ into irreducibles 
\be\label{E-forms-decomp}
\Lambda^1(M) \otimes E= (S_+^3\otimes S_-) \oplus (S_+\otimes S_-), \\ \nonumber
\Lambda^2(M) \otimes E = S_+^4 \oplus S_+^2 \oplus \R \oplus (S_+^2 \otimes S_-^2).
\ee
Note that these are standard statements about decomposition of various representations of ${\rm Spin}(4)$ into irreducibles. These fibrewise decompositions induce corresponding decompositions of the associated bundles.

\subsection{Algebra of $\Sigma$'s}

To obtain explicit formulas for the irreducible parts of $E$-valued differential forms, we need some identities satisfied by the 2-forms $\Sigma^i$. We will be using the index notation, similar to e.g. \cite{Spiro-G2}, which is most suited for the type of calculations that need to be done. 

First, one of the two indices of these differential forms can be raised with the metric (that they define), to convert these objects into those in ${\rm End}(T^*M)$. We then have a triple of such endomorphisms of the tangent bundle, satisfying the algebra of the imaginary quaternions
\be\label{algebra}
\Sigma^i_{\mu}{}^\alpha \Sigma^j_\alpha{}^\nu = - \delta^{ij} \delta_\mu{}^\nu + \epsilon^{ijk} \Sigma^k_\mu{}^\nu.
\ee
There are also useful relations
\be\label{sigma-sigma}
\Sigma^i_{\mu\nu} \Sigma^i_{\rho\sigma}= g_{\mu\rho} g_{\nu\sigma} - g_{\mu\sigma} g_{\nu\rho} + \epsilon_{\mu\nu\rho\sigma}, \\ 
\label{sigma-sigma-epsilon}
\epsilon^{ijk} \Sigma^j_{\mu\nu} \Sigma^k_{\rho\sigma}= -2\Sigma^i_{[\mu|\rho|} g_{\nu]\sigma} + 2\Sigma^i_{[\mu|\sigma|} g_{\nu]\rho} .
\ee
These can be proven using the expression (\ref{Sigma-can}) for $\Sigma^i$ in terms of a co-frame. 

\subsection{Decomposition of $\Lambda^1(M) \otimes E$}

Given an ${\rm SU}(2)$ structure $\Sigma^i$, we can define the following operator acting on $\Lambda^1(M) \otimes E$
\be\label{J-Sigma}
\Lambda^1(M) \otimes E\ni A^i_\mu \to J_1(A)_\mu^i := \epsilon^{ijk} \Sigma^j_\mu{}^\alpha A^k_\alpha.
\ee
A simple calculation using (\ref{algebra}) shows that 
\be
J_1^2 = 2\mathbb{I} + J_1.
\ee
This means that the eigenvalues of $J_\Sigma$ are $2, -1$. The eigenspaces of $J_1$ are precisely the irreducibles appearing in the first line in (\ref{E-forms-decomp}). It is also easy to check that objects of the form
\be
\xi^\alpha \Sigma^i_{\alpha\mu} \in \Lambda^1(M)\otimes E
\ee
are eigenvectors of eigenvalue $2$. We then have the following characterisation 
\be\label{lambda-E-4}
(\Lambda^1(M)\otimes E)_4 = S_+\otimes S_- =\{ \xi^\alpha \Sigma^i_{\alpha\mu}, \xi\in TM\}.
\ee
The space 
\be\label{lambda-E-8}
(\Lambda^1(M)\otimes E)_8 = S^3_+\otimes S_- 
\ee
can then be characterised as the orthogonal complement of (\ref{lambda-E-4}) in $\Lambda^1(M)\otimes E$. The subscripts $4,8$ in $(\Lambda^1(M)\otimes E)_{4,8}$ denote the dimensions of the irreducible components. Note that the fact that there are only two irreducible components of the intrinsic torsion is what makes our treatment different from that in \cite{Fowdar}. This is because we are using the decomposition with respect to ${\rm SO}(4)$ rather than with ${\rm SU}(2)$. As a result, only two irreducible representations appear. 

\subsection{${\rm GL}(4)$ orbit of $\Sigma$ in $\Lambda^2(M) \otimes E$}

To characterise some of the spaces appearing in the decomposition of $\Lambda^2(M) \otimes E$ we first consider the ${\rm GL}(4)$ orbit of the 2-forms $\Sigma^i$. Thus, we consider $E$-valued 2-forms of the form $h_{[\mu}{}^\alpha \Sigma^i_{ |\alpha|\nu]} $. Decomposing $h_{\mu\nu}\in{\rm GL}(4)$ into its symmetric and anti-symmetric parts, and noting that the anti-symmetric part is valued in $\Lambda^2(M)=S_+^2 \oplus S_-^2$, we get the following list of irreducibles appearing  
\be
h_{[\mu}{}^\alpha \Sigma^i_{ |\alpha|\nu]} \in S_+^2 \oplus \R \oplus (S_+^2 \otimes S_-^2) \subset \Lambda^2(M) \otimes E,
\ee
which is all spaces in the second line of (\ref{E-forms-decomp}) apart from $S_+^4$. These irreducibles in $\Lambda^2(M) \otimes E$ can then be characterised as the images of the map $h_{\mu\nu}\to h_{[\mu}{}^\alpha \Sigma^i_{ |\alpha|\nu]} $. Here and below the square brackets on cotangent space indices denote anti-symmetrisation, while the round brackets denote symmetrisation
\[
A_{[\mu} B_{\nu]} = \frac{1}{2}( A_\mu B_\nu - A_\nu B_\mu), \qquad A_{(\mu} B_{\nu)} = \frac{1}{2}( A_\mu B_\nu + A_\nu B_\mu).
\]

One can also act on the index $i$ of $\Sigma^i$ 2-forms with a ${\rm GL}(3)$ transformation, i.e., consider the orbit of $E$-valued 2-forms of the form $h^{ij} \Sigma^j_{\mu\nu}$. Decomposing the matrix $h^{ij}$ into symmetric and anti-symmetric parts, one finds the following list of irreducibles
\be
h^{ij} \Sigma^j_{\mu\nu} \in S_+^4 \oplus S_+^2 \oplus \R.
\ee

It is not hard to show that, in the opposite direction, given an object $B^i_{\mu\nu}\in \Lambda^2(M) \otimes E$, its irreducible parts can be extracted as follows
\be\label{irreducibles}
B^{(i}_{\alpha\beta} \Sigma^{j)\alpha\beta} - \frac{1}{3} \delta^{ij} B^{k}_{\alpha\beta} \Sigma^{k\alpha\beta} \in S_+^4, \\ \nonumber
\epsilon^{ijk} B^j_{\alpha\beta} \Sigma^{k\alpha\beta} \in S_+^2, \\ \nonumber
B^{k}_{\alpha\beta} \Sigma^{k\alpha\beta} \in \R, \\ \nonumber
B^i_{(\mu|\alpha|} \Sigma^{i\alpha}{}_{\nu)} + \frac{1}{4} g_{\mu\nu} B^i_{\alpha\beta} \Sigma^{i\alpha\beta} \in S_+^2 \otimes S_-^2 .
\ee

\subsection{Decomposition of $\Lambda^2(M) \otimes E$} We can also describe the irreducible subspaces of $\Lambda^2(M) \otimes E$ as eigenspaces of a certain operator in $E$-valued 2-forms, similar to how we used $J_1$ to decompose $\Lambda^1\otimes E$. Let us introduce the following operator
\be\label{J2}
J_2: \Lambda^2\otimes E\to \Lambda^2\otimes E, \qquad J_2(B)_{\mu\nu}^i = \epsilon^{ijk} \Sigma^j_{[\mu}{}^\alpha B_{|\alpha|\nu]}^k, \qquad B_{\mu\nu}^i\in \Lambda^2\otimes E.
\ee
A computation gives
\be
J_2^2(B)_{\mu\nu}^i = \frac{1}{2} B_{\mu\nu}^i + \frac{1}{2}\epsilon_{\mu\nu}{}^{\alpha\beta} B^i_{\alpha\beta} + \frac{1}{2} J_2(B)_{\mu\nu}^i + \frac{1}{2} \Sigma^i_{[\mu}{}^\alpha \Sigma^j_{\nu]}{}^\beta B^j_{\alpha\beta}, \\ \nonumber
J_2^3(B)_{\mu\nu}^i =  \frac{1}{2}\epsilon_{\mu\nu}{}^{\alpha\beta} B^i_{\alpha\beta} + 2 J_2(B)_{\mu\nu}^i +  \Sigma^i_{[\mu}{}^\alpha \Sigma^j_{\nu]}{}^\beta B^j_{\alpha\beta}, \\ \nonumber
J_2^4(B)_{\mu\nu}^i =  \frac{1}{2} B_{\mu\nu}^i+ \frac{3}{2}\epsilon_{\mu\nu}{}^{\alpha\beta} B^i_{\alpha\beta} + \frac{5}{2} J_2(B)_{\mu\nu}^i +  \frac{5}{2} \Sigma^i_{[\mu}{}^\alpha \Sigma^j_{\nu]}{}^\beta B^j_{\alpha\beta}.
\ee
This implies
\be
J_2^4 - 2J_2^3-J_2^2 + 2 J_2=0 \qquad \text{or} \qquad J_2 (J_2-2)(J_2-1)(J_2+1)=0,
\ee
which implies that the eigenvalues of $J_2$ are $2,1,-1,0$. 

To characterise the eigenspaces we consider an arbitrary $3\times 3$ matrix 
\[
M^{ij}=M_s^{ij}+M_a^{ij}, \qquad M_s^{ij}=M_s^{(ij)}, \qquad M_a^{ij} = M_a^{[ij]},
\] 
and compute
\be
J_2(M^{ij} \Sigma^j_{\mu\nu}) = {\rm Tr}(M) \Sigma^i_{\mu\nu}-M^{ji} \Sigma^j_{\mu\nu} = {\rm Tr}(M) \Sigma^i_{\mu\nu}-M_s^{ij} \Sigma^j_{\mu\nu}+ M_a^{ij} \Sigma^j_{\mu\nu}.
\ee
This means that the eigenspace of $J_2$ of eigenvalue $2$ is spanned by multiples of $\Sigma^i_{\mu\nu}$. The eigenspace of eigenvalue $1$ is $S_+^2$ spanned by $M_a^{ij} \Sigma^j_{\mu\nu}$. The eigenspace of eigenvalue $-1$ is $S_+^4$ spanned by $M_s^{ij} \Sigma^j_{\mu\nu}$ with ${\rm Tr}(M_s)=0$. 

We can also apply the operator $J_2$ to objects of the type $h_{[\mu}{}^\alpha \Sigma^i_{|\alpha|\nu]}$. We get
\be
J_2(h_{[\mu}{}^\alpha \Sigma^i_{|\alpha|\nu]}) = \frac{1}{2} h_\alpha{}^\alpha \Sigma^i_{\mu\nu}.
\ee
This means that the space $S_+^2\otimes S_-^2$ spanned by $h_{[\mu}{}^\alpha \Sigma^i_{|\alpha|\nu]}$ with tracefree $h_{\mu\nu}$ is eigenspace of $J_2$ of eigenvalue $0$.

All in all, we get
\be
\Lambda^2\otimes E= (\Lambda^2\otimes E)_5\oplus (\Lambda^2\otimes E)_3\oplus (\Lambda^2\otimes E)_1\oplus (\Lambda^2\otimes E)_9.
\ee
The last space here is $(\Lambda^2\otimes E)_9= \Lambda^-\otimes E$.

\section{Intrinsic torsion and the associated curvature}
\label{sec:torsion}

\subsection{Intrinsic torsion}

From general principles, it follows that the torsion of a $G$-structure should be described by a an object valued in $\Lambda^1(M)\otimes \mathfrak{g}^\perp$, which in our case is $\Lambda^1(M)\otimes E$. At the same time, the intrinsic torsion quantifies non-integrability of the $G$-structure, and thus the failure of the tensors defining this $G$-structure to be parallel with respect to the Levi-Civita connection. Thus, we expect that $\nabla_\mu \Sigma^i_{\alpha\beta}$ should be expressible via the intrinsic torsion $A^i_\mu \in \Lambda^1(M)\otimes E$. The following proposition is a statement to this effect
\begin{theorem} There exists a set of objects $A_\mu^i\in \Lambda^1(M)\otimes E$ such that
\be\label{torsion}
\nabla_\mu \Sigma^i_{\alpha\beta} = - \epsilon^{ijk} A^j_\mu \Sigma^k_{\alpha\beta}.
\ee
\end{theorem}
\begin{proof} Comparing the right-hand side of the formula (\ref{torsion}) with the set of objects that appear in (\ref{irreducibles}), we see that the statement is that there are no $S_+^4, \R, S_+^2 \otimes S_-^2$ irreducible components in $X^\mu \nabla_\mu \Sigma^i_{\alpha\beta}\in \Lambda^2(M)\otimes E, \forall X^\mu \in TM$. The $S_+^4, \R$ components are extracted as
\be
2\Sigma^{(i|\alpha\beta|} \nabla_\mu \Sigma^{j)}_{\alpha\beta} = \nabla_\mu (\Sigma^{i|\alpha\beta|}  \Sigma^{j}_{\alpha\beta}) = 4 \nabla_\mu \delta^{ij} = 0.
\ee
Here we have used the fact that the operation of raising-lowering of the indices commutes with $\nabla_\mu$. Similarly, the $S_+^2 \otimes S_-^2$ component is extracted as
\be
2\Sigma^i_{(\mu}{}^\alpha \nabla_\rho \Sigma^i_{\nu)\alpha} = \nabla_\rho \Sigma^i_{\mu}{}^\alpha  \Sigma^i_{\nu\alpha} = 3 \nabla_\rho g_{\mu\nu} = 0.
\ee
This shows that no undesired components are present in $\nabla_\mu \Sigma^i_{\alpha\beta}$ and that (\ref{torsion}) holds. 
\end{proof} 
Below, in equaiton (\ref{A-gauge-transf}), we verify that the objects $A^i_\mu$ transform under ${\rm SO}(3)$ transformations of $\Sigma$'s as an ${\rm SO}(3)$-connection. This proves Theorem \ref{thm:nabla-sigma} of the Introduction. The fact that the intrinsic torsion assembles itself into an ${\rm SO}(3)$ connection does not have analogs in the case of general G-structures, this is the property of $H$-covariant $G$-structures.

\subsection{Bianchi identity}

Establishing a version of the formula (\ref{torsion}) is one of the more laborious parts of the analysis of a non-integrable $G$-structure. The rest of the analysis is much more algorithmic. We take another covariant derivative and anti-symmetrise to get
\be\label{bianchi}
2 R_{\mu\nu[\rho}{}^{\alpha} \Sigma^i_{|\alpha| \sigma]} = 2 \nabla_{[\mu} \nabla_{\nu]} \Sigma^i_{\rho\sigma} = - 2\epsilon^{ijk} (\nabla_{[\mu} A^j_{\nu]} \Sigma^k_{\rho\sigma}+
A^j_{[\nu} \nabla_{\mu]} \Sigma^k_{\rho\sigma})= \\ \nonumber
- 2\epsilon^{ijk} (\nabla_{[\mu} A^j_{\nu]} \Sigma^k_{\rho\sigma}+
A^j_{[\mu} \epsilon^{klm} A^l_{\nu]} \Sigma^l_{\rho\sigma}) = - \epsilon^{ijk} F^j_{\mu\nu} \Sigma^k_{\rho\sigma},
\ee
where
\be\label{F-curv}
F^i_{\mu\nu} := 2\partial_{[\mu} A^i_{\nu]} +\epsilon^{ijk} A^j_\mu A^k_\nu
\ee
is the curvature of the connection $A^i_\mu$. Importantly, we observe that, in the case of ${\rm SU}(2)$ structures in dimension four, the intrinsic torsion assembles itself into an $\mathfrak{su}(2)$-valued one-form, or an ${\rm SO}(3)$ connection. This connection gives rise to its curvature 2-form (\ref{F-curv}). 

The left-hand side in (\ref{bianchi}) is just the projection of the Riemann tensor that is valued in the symmetric second power $\Lambda^2(M)\otimes_S\Lambda^2(M)$ onto $E$ with respect to the second pair of indices. So, there is no loss of information if we multiply both sides of (\ref{bianchi}) with $\epsilon^{ijk}\Sigma^{j\rho\sigma}$ to get
\be\label{Riemann-F}
 R_{\mu\nu}{}^{\rho\sigma} \Sigma^k_{\rho\sigma} = 2F^k_{\mu\nu}.
\ee
This is the most useful form of the "Bianchi identity" (\ref{bianchi}), using the terminology of \cite{Spiro-G2}. In words, the self-dual part of the Riemann curvature $R_{\mu\nu\rho\sigma}$ with respect to the pair of indices $\rho\sigma$ equals a multiple of the curvature tensor $F^i_{\mu\nu}$, which is also the curvature of the intrinsic torsion $A^i_\mu$. 

\subsection{Ricci tensor} 

We can extract the Ricci tensor from (\ref{Riemann-F}) via
\be
\Sigma^i_{\mu}{}^\alpha R_{\alpha\nu \rho\sigma} \Sigma^i_{\rho\sigma} = ( g_{\mu\rho} g^\alpha{}_{\sigma} - g_{\mu\sigma} g^\alpha{}_{\rho}+\epsilon_\mu{}^\alpha{}_{\rho\sigma}) R_{\alpha\nu \rho\sigma} = -2 R_{\mu\nu}, 
\ee
where we used (\ref{sigma-sigma}). On the other hand, applying this to the right-hand side of (\ref{Riemann-F}) we get
\be
R_{\mu\nu} = -\Sigma^i_{\mu}{}^\alpha F^i_{\alpha\nu}.
\ee
Thus, in particular, 
\be
s = \Sigma^{i\mu\nu} F^i_{\mu\nu}.
\ee
The inverse of the formula for the Ricci curvature is
\be
F^i_{\mu\nu} = \Psi^{ij} \Sigma^j_{\mu\nu} - \frac{s}{6} \Sigma^i_{\mu\nu} + R_{[\mu}{}^\alpha \Sigma^i_{|\alpha|\nu]}.
\ee
Here $\Psi^{ij}$ is the matrix of components of the chiral half of the Weyl curvature. Using $R_{\mu\nu} = \tilde{R}_{\mu\nu} + \frac{1}{4} s g_{\mu\nu}$, where $\tilde{R}_{\mu\nu}$ is the tracefree part of the Ricci curvature, we can also rewrite this as 
\be\label{F-in-terms-curvature}
F^i_{\mu\nu} = \Psi^{ij} \Sigma^j_{\mu\nu} + \frac{s}{12} \Sigma^i_{\mu\nu} + \tilde{R}_{[\mu}{}^\alpha \Sigma^i_{|\alpha|\nu]}.
\ee
The first two terms here are self-dual as 2-forms, the last is anti-self-dual. This proves the Proposition \ref{prop:curv} of the Introduction. 

\subsection{Einstein condition}

As is well-known, the Riemann curvature viewed as a symmetric endomorphism of $\Lambda^2(M)$, decomposed into its self-dual and anti-self-dual blocks, reproduces the decomposition into Ricci and Weyl parts of the curvature. This is most usefully captured by the following matrix representation
\be\label{riemann}
\text{Riemann} = \left( \begin{array}{cc} W^+ + R & Rc^0 \\ Rc^0 & W^- + R \end{array}\right).
\ee
Here $W^\pm$ are the two chiral halves of the Weyl curvature, and $Rc^0$ is the tracefree part of the Ricci tensor. The trace part is denoted by $R$ and is the scalar curvature. The first row of this matrix is the self-dual part of Riemann with respect to the second pair of indices, and the second row is the anti-self-dual part. Similarly, the first (second) column is the self-dual (anti-self-dual) part of Reimann with respect to the first pair of indices. We thus see that the curvature $F^i_{\mu\nu}$ of the intrinsic torsion encodes precisely the first row of the matrix (\ref{riemann}), and thus the self-dual part $W^+$ if the Weyl curvature, as well as all of the Ricci curvature. 

It is now clear that the Einstein condition can be encoded as one on the curvature $F^i_{\mu\nu}$. The condition that $F^i_{\mu\nu}$ is self-dual as a 2-form is equivalent to the condition that the tracefree part $Rc^0$ of Ricci vanishes
\be
F^i \in \Lambda^+ \Leftrightarrow Rc^0=0.
\ee
The scalar curvature can then be set to any desired value by imposing a condition on the self-dual part of $F^i$. All in all, Einstein equations are most usefully stated as the condition
\be\label{einstein}
F^i_{\mu\nu} = \left( \Psi^{ij} + \frac{\Lambda}{3} \delta^{ij}\right) \Sigma^j_{\mu\nu}.
\ee
Here $\Psi^{ij}$ is an arbitrary symmetric tracefree $3\times 3$ matrix, which encodes the $W^+$ part of the curvature, and is not constrained by the Einstein equations. The constant $\Lambda=4s$ is a multiple of the scalar curvature $s$. The equation (\ref{einstein}) is equivalent to $R_{\mu\nu}=\Lambda g_{\mu\nu}$ Einstein condition. 

\section{Linearised analysis}
\label{sec:linearised}

The purpose of this section is to consider perturbations of ${\rm SU}(2)$ structures, around the flat space $\R^4$, and construct the most general diffeomorphism invariant Lagrangian for such perturbations. This linearised story provides a very good intuition for the non-linear story in the next section. In this section $\Sigma^i$ is chosen to be the canonical basis of self-dual 2-forms for the flat space $\R^4$. In particular, all partial derivatives of the objects $\Sigma^i_{\mu\nu}$ vanish. 

\subsection{Perturbation of a ${\rm SU}(2)$ structure} 

The tangent space to the ${\rm GL}(4)$ orbit of $\Sigma^i$ contains irreducible representations $(\Lambda^2\otimes E)_{3+1+9}$. We can parametrise these spaces as
\be
(\Lambda^2\otimes E)_{1+9} \ni 2h_{[\mu}{}^\alpha \Sigma_{|\alpha|\nu]}^i, \qquad (\Lambda^2\otimes E)_{3} \ni 2 \epsilon^{ijk} \Sigma^j_{\mu\nu}\xi^k,
\ee
with $h_{\mu\nu}$ being a symmetric tensor and $\xi^i\in E$. The role of the numerical factors chosen is to simplify some formulas that follow. This means that perturbations of $\Sigma^i_{\mu\nu}$, which we denote by $\delta \Sigma^i_{\mu\nu}:=\sigma^i_{\mu\nu}$ can be parametrised as
\be
\sigma^i_{\mu\nu} = 2h_{[\mu}{}^\alpha \Sigma_{|\alpha|\nu]}^i + 2 \epsilon^{ijk}  \Sigma^j_{\mu\nu} \xi^k.
\ee
The inverse is given by
\be
h_{\mu\nu} = - \frac{1}{2} \sigma^i_{(\mu}{}^\alpha \Sigma_{|\alpha|\nu)}^i - \frac{1}{12} \eta_{\mu\nu} \Sigma^{i\rho\sigma} \sigma^i_{\rho\sigma}, \\ \nonumber
\xi^i = -\frac{1}{16}\epsilon^{ijk} \Sigma^{j\mu\nu} \sigma^k_{\mu\nu}.
\ee
Here $\eta_{\mu\nu}$ is the standard (flat) metric on $\R^4$. 

\subsection{Transformation properties under diffeomorphisms}

Let us consider a background of a constant triple of 2-forms $\Sigma^i$. The diffeomorphisms act $\delta_X \Sigma^i = {\mathcal L}_X \Sigma^i = i_X d\Sigma^i + d i_X \Sigma^i$. In the case of a constant triple of 2-forms we get $\delta_X \sigma^i = d i_X \Sigma^i$. In index notation
\be
\delta_X \sigma^i_{\mu\nu} = 2 \partial_{[\mu} X^\alpha \Sigma^i_{|\alpha| \nu]}.
\ee
This means that 
\be\label{diffeo-lin}
\delta_X h_{\mu\nu} = \partial_{(\mu} X_{\nu)}, \qquad \delta_X \xi^i = \frac{1}{4} \Sigma^{i\mu\nu} \partial_\mu X_\nu.
\ee

Our index notation conventions for differential forms are captured by 
\be
\omega = \frac{1}{k!} \omega_{\mu_1 \ldots \mu_k} dx^{\mu_1} \ldots dx^{\mu_k}.
\ee
Here $\omega$ is a k-form. This means that the exterior derivative operator acts as
\be
d\omega = \frac{1}{k!} \partial_{\mu} \omega_{\mu_1 \ldots \mu_k} dx^\mu dx^{\mu_1} \ldots dx^{\mu_k},
\ee
so that
\be
(d\omega)_{\mu_1 \ldots \mu_{k+1}} = (k+1) \partial_{[\mu} \omega_{\mu_1 \ldots \mu_k]},
\ee
where the $[\cdot]$ denotes anti-symmetrisation.

\subsection{Transformation properties under ${\rm SO}(3)$}

In addition to diffeomorphisms, we can also consider how quantities transform under the ${\rm SO}(3)$ transformations that rotate $\Sigma^i$. The infinitesimal version of these transformations is
\be
\delta_\phi \sigma^i_{\mu\nu} = 2 \epsilon^{ijk}  \Sigma^j_{\mu\nu}\phi^k,
\ee
where $\phi^i \in \Gamma(E)$. Under these transformations
\be\label{su2-action}
\delta_\phi h_{\mu\nu} = 0, \qquad \delta_\phi \xi^i = \phi^i.
\ee

\subsection{Second order action functional}

We now determine the most general diffeomorphism invariant action functional that can be written in terms of fields $h_{\mu\nu}$ and $\xi^i$, subject to the transformation properties (\ref{diffeo-lin}). We first write the general linear combination of all possible terms. The types of terms are dictated by simple representation theory. First, we can write the most general linear combination of terms that can be constructed solely from $h_{\mu\nu}$. This is standard and independent of the dimension. 
\begin{proposition}
Let $h_{\mu\nu}\in {\rm Sym}^2(T^*M)$ be a symmetric rank two tensor. The most general second order in derivatives of $h_{\mu\nu}$ Lagrangian is given by
\be\label{lin-h-Lagr}
\frac{\rho}{2} (\partial_\mu h_{\nu\rho})^2 + \frac{\alpha}{2} (\partial_\mu h)^2 - \beta h \partial^\mu \partial^\nu h_{\mu\nu} - \gamma (\partial^\mu h_{\mu\nu})^2,
\ee
where $\alpha,\beta,\gamma,\rho$ are arbitrary (real) parameters and $h=\eta^{\mu\nu} h_{\mu\nu}$. 
\end{proposition}
\begin{proof} This proposition holds in any dimension, but for our purposes it is sufficient to prove it in four dimensions, where we can decompose all arising representations into symmetric powers of the fundamental representations $S_\pm$ of the two ${\rm SU}(2)$'s in ${\rm Spin}(4)={\rm SU}(2)\times{\rm SU}(2)$. We have
\be
h_{\mu\nu} \in {\rm Sym}^2( S_+\otimes S_-) = S_+^2\otimes S_-^2 \oplus \R.
\ee
We are interested in determining the Lagrangian modulo boundary terms that arise by integration by parts, and are assumed to vanish. This means that we need to consider two instances of the partial derivative $\partial_\mu \partial_\nu$ as in object in 
\be\label{partial2}
\partial_\mu \partial_\nu \in {\rm Sym}^2( S_+\otimes S_-)= S_+^2\otimes S_-^2 \oplus \R.
\ee
We now take the product of two copies of $h_{\mu\nu}$. This is in the space
\be\label{h2}
{\rm Sym}^2( S_+^2\otimes S_-^2 \oplus C^\infty(M)) = {\rm Sym}^2( S_+^2\otimes S_-^2) \oplus S_+^2\otimes S_-^2 \oplus \R.
\ee
The Lagrangian will contract two copies of $h_{\mu\nu}$ with two partial derivatives $\partial_\mu \partial_\nu$ to obtain a scalar. We are thus interested in the factors in (\ref{h2}) that are in the two spaces appearing in (\ref{partial2}). The decomposition 
\be\label{sym2-h}
{\rm Sym}^2( S_+^2\otimes S_-^2) = S_+^4 \otimes S_-^4 \oplus S_+^4\oplus S_-^4 \oplus S_+^2 \otimes S_-^2 \oplus \R
\ee
contains exactly two such factors. Together with the two more such factors in (\ref{h2}) this shows that there are four different possible terms in the Lagrangian we are trying to construct. It is not difficult to see that they are precisely the terms we have written in (\ref{lin-h-Lagr}). Indeed, the first two terms in (\ref{lin-h-Lagr}) contain the derivative contracted with itself, and thus corresponds to the $\R$ factor in (\ref{partial2}). There are two $\R$ factors in (\ref{h2}), one given by $h^2$, the other $(h_{\mu\nu})^2$. The third term in (\ref{lin-h-Lagr}) corresponds to the $S_+^2\otimes S_-^2$ factor in (\ref{partial2}) contracting with the $S_+^2\otimes S_-^2$ second factor in the right-hand-side of (\ref{h2}). The last term corresponds to $S_+^2\otimes S_-^2$ factor in (\ref{partial2}) contracting with the $S_+^2\otimes S_-^2$ factor in (\ref{sym2-h}). 
\end{proof}

We then need to determine all possible terms involving two copies of $\xi^i$, as well as $h \xi$ terms. The field $h_{\mu\nu}$ lives in $S_+^2\otimes S_-^2$ as well as $\R$. The field $\xi^i$ is in $S_+^2$. We have the following tensor products
\be
(S_+^2\otimes S_-^2) \otimes S_+^2 = (S_+^4\otimes S_-^2) \oplus (S_+^2\otimes S_-^2) \oplus S_-^2, \\ \nonumber
{\rm Sym}^2 (S_+^2)   = S_+^4 \oplus \R.
\ee
We need to combine these irreducible pieces with those arising from the symmetrised product of two partial derivatives (\ref{partial2}). This makes it clear that the only term that can be constructed from two copies of $\xi^i$ is $(\partial_\mu \xi^i)^2$. There is also just a single term that can be constructed from $h_{\mu\nu}$ and $\xi^i$, which is 
\be
(\partial^\mu h_{\mu\nu}) (\partial^\alpha \xi^i ) \Sigma^i_{\alpha}{}^\nu.
\ee
This gives the following most general Lagrangian
\be\nonumber
{\mathcal L} = \frac{\rho}{2} (\partial_\mu h_{\nu\rho})^2 + \frac{\alpha}{2} (\partial_\mu h)^2 - \beta h \partial^\mu \partial^\nu h_{\mu\nu} - \gamma (\partial^\mu h_{\mu\nu})^2+ \frac{\lambda}{2} (\partial_\mu \xi^i)^2 + \mu (\partial^\mu h_{\mu\nu}) (\partial^\alpha \xi^i ) \Sigma^i_{\alpha}{}^\nu.
\ee

\subsection{Diffeomorphism invariant Lagrangian}

We now consider the effect of diffeomorphisms, and prove the following
\begin{proposition}
The most general diffeomorphism invariant that is second order in derivatives is given by
\be\label{L-diff-inv}
{\mathcal L} = \rho {\mathcal L}_{GR} + \mu {\mathcal L}',
\ee
with
\be
{\mathcal L}_{GR} = \frac{1}{2} (\partial_\mu h_{\nu\rho})^2 - \frac{1}{2} (\partial_\mu h)^2 -  h \partial^\mu \partial^\nu h_{\mu\nu} - (\partial^\mu h_{\mu\nu})^2, \\ \nonumber
{\mathcal L}'=  -\frac{1}{4} (\partial_\mu h)^2 - \frac{1}{2} h \partial^\mu \partial^\nu h_{\mu\nu} - \frac{1}{4} (\partial^\mu h_{\mu\nu})^2-  (\partial_\mu \xi^i)^2 + (\partial^\mu h_{\mu\nu}) (\partial^\alpha \xi^i ) \Sigma^i_{\alpha}{}^\nu.
\ee
\end{proposition}
\begin{proof}
We calculate the effect of a diffeomorphism on $\mathcal L$, integrating by parts when necessary, and assuming that the arising boundary terms vanish. We use the symbol $\approx$ to denote equality modulo integration by parts. We have
\be\nonumber
\delta_X {\mathcal L} \approx (\rho-\gamma+\frac{\mu}{4}) \partial^2 (\partial^\mu h_{\mu\nu}) X^\nu - ( \alpha + \beta) \partial^2 h (\partial X) + (-\beta+\gamma + \frac{\mu}{4}) (\partial X) (\partial^\mu\partial^\nu h_{\mu\nu}) \\ \nonumber 
-(  \frac{\lambda}{4} + \frac{\mu}{2}) \partial^2 \xi^i \Sigma^{i\mu\nu} \partial_\mu X_\nu. 
\ee
Equating the coefficients in front of the independent terms to zero, and parametrising the solution by $\rho,\mu$ we have
\be
\beta= - \alpha = \rho+\frac{\mu}{2}, \quad \gamma= \rho + \frac{\mu}{4}, \qquad \lambda = - 2\mu. 
\ee
\end{proof}
We note that this is a very similar story to what happens in the case of ${\rm Spin}(7)$ structures in eight dimensions, see \cite{Krasnov:2024lcl}. In that context, as here, the most general diffeomoprhism invariant Lagrangian is also given by a linear combination of two terms. 

\subsection{Lagrangian that is also ${\rm SO}(3)$ gauge invariant}

The Lagrangian (\ref{L-diff-inv}) is diffeomoprhism invariant (modulo integration by parts). It is also invariant under global (i.e. rigid) ${\rm SO}(3)$ rotations acting on $E$. However, it is clear that it is possible to demand also local ${\rm SO}(3)$ invariance. From (\ref{su2-action}) we see that these transformations act only on $\xi^i$. It is clear that the Lagrangian ${\mathcal L}'$ is not invariant under such local transformations. Therefore, only ${\mathcal L}_{GR}$ is both diffeomorphism and ${\rm SO}(3)$ gauge invariant. It is therefore to be expected that there exists a unique non-linear Lagrangian for $\Sigma^i_{\mu\nu}$, which is second order in derivatives, and both diffeomorphism and ${\rm SO}(3)$ gauge invariant. We can also expect this non-linear action to have critical points that are Einstein metrics. This is exactly what happens, as we shall now verify. 

\section{Action functionals}
\label{sec:actions}

In preparation to the construction of the action, we will first show that the intrinsic torsion is completely determined by the exterior derivative $d\Sigma^i$. 

\subsection{Torsion in terms of $d\Sigma^i$}

In (\ref{torsion}) we have related the torsion $A^i_\mu$ to the covariant derivative $\nabla_\mu \Sigma^i_{\alpha\beta}$ of the 2-forms $\Sigma^i$. We now explain that the knowledge of the exterior derivative is sufficient
\begin{theorem} The intrinsic torsion is determined by the exterior derivatives of the 2-forms $\Sigma^i$. Specifically, we have
\be\label{torsion-d-Sigma}
A= \frac{1}{4}( J_1 - \mathbb{I})({}^* d\Sigma^i),
\ee
where $J_1$ is the operator (\ref{J-Sigma}) and ${}^* d\Sigma^i$ is the Hodge dual of the 3-form $d\Sigma^i$. 
\end{theorem}
\begin{proof} We project the equation (\ref{torsion}) to the space of 3-forms, anti-symmetrising over all 3 indices. We have
\be
\partial_{[\nu} \Sigma^i_{\alpha\beta]} = - \epsilon^{ijk} A^j_{[\nu} \Sigma^k_{\alpha\beta]}.
\ee
We can write this in index-free differential form notations as
\be\label{compat}
d\Sigma^i + \epsilon^{ijk} A^j \Sigma^k=0.
\ee
To solve this, we multiply with the $\epsilon$ tensor and use the self-duality of $\Sigma^i_{\mu\nu}$
\be
\epsilon^{\mu\nu\alpha\beta}\partial_{\nu} \Sigma^i_{\alpha\beta}  = - 2 \epsilon^{ijk} A^j_{\nu} \Sigma^{k\mu\nu} =  2(J_1(A))^{i\mu},
\ee
where $J_1$ is the operator on $\Lambda^1(M)\otimes E$ that was introduced in (\ref{J-Sigma}). We can write this in an index-free way as
\be
{}^* d\Sigma^i =  2J_1(A).
\ee
The $J_1$ operator is invertible, with the inverse given by
\be
J_1^{-1} =\frac{1}{2}( J_1 - \mathbb{I}).
\ee
This establishes (\ref{torsion-d-Sigma}). This theorem is also proven in \cite{Fowdar}, see formulas (26)-(28), but our answer is notably more compact. 
\end{proof}

We have an immediate well-known corollary.
\begin{corollary} An ${\rm SU}(2)$-structure is integrable if and only if $d\Sigma^i=0$. 
\end{corollary}
Indeed, an ${\rm SU}(2)$-structure is integrable if and only if its intrinsic torsion vanishes, which, in view of (\ref{torsion-d-Sigma}), is clearly equivalent to the conditions $d\Sigma^i=0$.

\subsection{Calculation of $A$}
 
The aim of this subsection is to establish a useful explicit formula for the intrinsic torsion in terms of $d\Sigma^i$. We have
 \be
( *d\Sigma)^i_\mu = \epsilon_{\mu}{}^{\alpha\beta\gamma} \partial_\alpha \Sigma^i_{\beta\gamma}, 
\ee
and
\be
A^i_\mu = \frac{1}{4} (J_1 -\mathbb{I})(*d\Sigma)^i_\mu = -\frac{1}{4} \epsilon_{\mu}{}^{\alpha\beta\gamma} \partial_\alpha \Sigma^i_{\beta\gamma} + \frac{1}{4} \epsilon^{ijk} \Sigma^j_\mu{}^\alpha \epsilon_{\alpha}{}^{\beta\gamma\delta} \partial_\beta \Sigma^k_{\gamma\delta}.
\ee
We can can simplify the last term using
\be\label{epsilon-Sigma}
\epsilon^{\mu\nu\rho\sigma}  \Sigma^i_{\alpha\sigma}= 3 \delta_\alpha^{[\rho} \Sigma^{i \mu\nu]}.
\ee
This gives
\be\label{A-dS}
A_\mu^i  =  -\frac{1}{4} \epsilon_{\mu}{}^{\alpha\beta\gamma} \partial_\alpha \Sigma^i_{\beta\gamma} - \frac{1}{4} \epsilon^{ijk} \Sigma^{j\alpha\beta} \partial_\mu \Sigma^k_{\alpha\beta} -\frac{1}{2} \epsilon^{ijk} \Sigma^{j\alpha\beta} \partial_\beta \Sigma^k_{\mu\alpha}.
\ee
This gives an explicit formula for the intrinsic torsion of an ${\rm SU}(2)$-structure in terms of derivatives of $\Sigma$'s. This formula is useful if one wants to evaluate a torsion squared explicitly in terms of derivatives of the 2-forms $\Sigma^i$. 

\subsection{Bianchi identity} 

The aim of this subsection is to establish a useful consequence of the equation (\ref{compat}) for the intrinsic torsion. We take the exterior derivative of this equation to get
\be
\epsilon^{ijk} dA^j \Sigma^k - \epsilon^{ijk} A^j d\Sigma^k=0.
\ee
We now substitute $d\Sigma^k$ from (\ref{compat}) as $d\Sigma^k =- \epsilon^{klm} A^l \Sigma^m$. We then use $A^j A^l = (1/2) \epsilon^{jls} \epsilon^{spq} A^p A^q$
to rewrite
\be
\epsilon^{ijk} A^j \epsilon^{klm} A^l \Sigma^m = \epsilon^{ijk} (\frac{1}{2} \epsilon^{jlm} A^l A^m) \Sigma^k.
\ee
All in all, we get
\be\label{bianchi-F}
\epsilon^{ijk} F^j \Sigma^k =0,
\ee
where $F^i$ is the curvature 
\be
F^i = dA^i + \frac{1}{2} \epsilon^{ijk} A^j A^k.
\ee
We note that (\ref{bianchi-F}) can be interpreted as the statement that there is no $S_+^2$ component in the decomposition of the $F\in \Lambda^2(M) \otimes E$ into its irreducible components. 

\subsection{Transformation properties under diffeomorphisms}

The aim of this subsection is to establish how the intrinsic torsion transforms under diffeomorphisms. The ${\rm SU}(2)$ structure transforms under the diffeomoprhisms according to 
$\delta_X \Sigma^i = d i_X \Sigma^i + i_X d\Sigma^i$. The transformation property of the intrinsic torsion that solves (\ref{compat}) can be determined as follows. Taking the variation of (\ref{compat}) we have
\be\label{delta-compat}
d (i_X d\Sigma^i) + \epsilon^{ijk} \delta_X A^j \Sigma^k + \epsilon^{ijk} A^j ( di_X\Sigma^i + i_X d\Sigma^i)=0.
\ee
We can also insert the vector field $X$ into (\ref{compat}) to get
\[
i_X d\Sigma^i + \epsilon^{ijk} (i_X A^j) \Sigma^k - \epsilon^{ijk} A^j i_X \Sigma^k=0.
\]
Substituting $i_X d\Sigma^i$ from here into (\ref{delta-compat}) we have
\be\label{delta-compat-1}
d (\epsilon^{ijk} A^j i_X \Sigma^k-\epsilon^{ijk} (i_X A^j) \Sigma^k) + \epsilon^{ijk} \delta_X A^j \Sigma^k + \epsilon^{ijk} A^j  di_X\Sigma^k  \\ \nonumber
+\epsilon^{ijk} A^j (\epsilon^{klm} A^l i_X \Sigma^m-\epsilon^{klm} (i_X A^l) \Sigma^m)=0.
\ee
The terms in the first line become
\[
\epsilon^{ijk} dA^j i_X \Sigma^k-\epsilon^{ijk} d(i_X A^j) \Sigma^k + \epsilon^{ijk} (i_X A^j) \epsilon^{klm} A^l \Sigma^m + \epsilon^{ijk} \delta_X A^j \Sigma^k ,
\]
where we have used (\ref{compat}) again. 
The first term in the second line can also be simplified. We again use $A^j A^l = (1/2) \epsilon^{jls} \epsilon^{spq} A^p A^q$ to get
\[
\epsilon^{ijk} A^j \epsilon^{klm} A^l i_X \Sigma^m =\epsilon^{ijk} (\frac{1}{2} \epsilon^{jlm} A^l A^m) i_X \Sigma^k.
\]
This means that (\ref{delta-compat-1}) can be rewritten as
\[
\epsilon^{ijk} F^j i_X \Sigma^k + \epsilon^{ijk}( \delta_X A^j -d(i_X A^j) ) \Sigma^k + \epsilon^{ijk} (i_X A^j) \epsilon^{klm} A^l \Sigma^m  
-\epsilon^{ilk} A^l \epsilon^{kjm} (i_X A^j) \Sigma^m=0,
\]
where we changed the names of the dummy indices suggestively.
The last two terms can be simplified using the identity
\be\label{bianchi-epsilon}
\epsilon^{ijk}\epsilon^{klm}+\epsilon^{ilk}\epsilon^{kmj} + \epsilon^{imk}\epsilon^{kjl}=0.
\ee
This gives
\[
\epsilon^{ijk} F^j i_X \Sigma^k + \epsilon^{ijk}( \delta_X A^j -d(i_X A^j) - \epsilon^{jpq} A^p (i_X A^q)) \Sigma^k =0.
\]
We can finally insert $X$ into (\ref{bianchi-F}) to rewrite the first term here as $- \epsilon^{ijk} i_X F^j  \Sigma^k$. Overall, this produces terms that are all of the type of operator $J_1$ acting on an $E$-valued 1-form. The operator $J_1$ is invertible, which allows us to write
\be\label{torsion-diffeo}
\delta_X A^i = d(i_X A^i) + \epsilon^{ijk} A^j (i_X A^k) + i_X F^i.
\ee
The first two terms here assemble into the covariant derivative of $i_X A^i$ computed using the connection $A^i$. The last term is the insertion of $X$ into the curvature $F^i$. This is of course as expected, because using the formula for $F^i$ and noting a cancelation this can be rewritten as
\[
\delta_X A^i = d i_X A^i + i_X d A^i.
\]
This confirms that the torsion transforms covariantly under diffeomorphisms, and gives a very useful formula (\ref{torsion-diffeo}). 

\subsection{Transformation properties under ${\rm SO}(3)$ gauge transformations}

Let us also determine how the torsion transforms under the local ${\rm SO}(3)$ gauge transformations $\delta_\phi \Sigma^i = \epsilon^{ijk} \phi^k \Sigma^k$. Taking the variation of (\ref{compat}) we have
\[
d (\epsilon^{ijk} \phi^j \Sigma^k) + \epsilon^{ijk} \delta_\phi A^j \Sigma^k + \epsilon^{ijk} A^j \epsilon^{klm} \phi^l \Sigma^m=0.
\]
The first term gives a contribution containing $d\phi^i$, as well as one with $d\Sigma^i$. The latter can be transformed using (\ref{compat}). This gives
\[
 \epsilon^{ijk} (\delta_\phi A^j + d\phi^j)\Sigma^k + \epsilon^{ijk} A^j \epsilon^{klm} \phi^l \Sigma^m- \epsilon^{ijk} \phi^j \epsilon^{klm} A^l \Sigma^m=0.
 \]
 The last two terms can again be transformed using (\ref{bianchi-epsilon}). This puts all terms in the same form of $J_1$ acting on an $E$-valued 1-form. Because $J_1$ is invertible we get
 \begin{equation}\label{A-gauge-transf}
 \delta_\phi A^i = - d\phi^i - \epsilon^{ijk} A^j \phi^k,
 \end{equation}
 which is the usual gauge transformation with parameter $-\phi^i$. This confirms that the intrinsic torsion $A^i$ is an ${\rm SO}(3)$-connection. 
 This implies that the curvature $F^i$ transforms covariantly
 \[
 \delta_\phi F^i = \epsilon^{ijk} \phi^j F^k.
 \]

\subsection{Diffeomorphism invariant action}

We have confirmed that the torsion transforms covariantly under diffeomorphisms. This means that any action that is schematically of the type $\int A^2$ is diffeomorphism invariant. Now, the representation theoretic decomposition (\ref{lambda-E-4}), (\ref{lambda-E-8}) of $A\in \Lambda^1(M) \otimes E$ shows that there are two irreducible components of the intrinsic torsion. This means that there are only two quadratic invariants that can be constructed from $A$. One can always take as a basis of such invariants the quantities $(A_\mu^i)^2$ and $A_\mu^i J_1(A)^{i\mu}$. It can be confirmed that the linearisation of this general diffeomorphism invariant action coincides with the linearised action (\ref{L-diff-inv}). This establishes that the most general diffeomorphism invariant action for $\Sigma$, which is second order in derivatives and is also invariant under global ${\rm SO}(3)$ rotations of $\Sigma$, is given by a linear combination of $\int(A_\mu^i)^2$ and $\int A_\mu^i J_\Sigma(A)^{i\mu}$. Note that already the requirement of global  ${\rm SO}(3)$-invariance implies a drastic reduction in the number of possible torsion squared terms. Indeed, the reference \cite{Fowdar} shows that, if no requirement of the ${\rm SO}(3)$-invariance is implied, there are 15 such terms. 

\subsection{Diffeomorphism and ${\rm SO}(3)$ invariant action}

Let us now impose the requirement that the action is both diffeomorphism and locally ${\rm SO}(3)$ gauge invariant. At the linearised level, we have seen that this has the effect that only one of the two diffeomorphism invariant terms survives, and one gets linearised Einstein-Hilbert action. It is clear that from the two terms $A^2$ and $A J_1(A)$ the first one is not gauge invariant. Using physics terminology, this term is a mass term for the connection, which cannot be gauge invariant. Let us discuss the other term. We claim that it is both diffeomorphism and ${\rm SO}(3)$ invariant. To see this, it is best to rewrite it using some integration by parts identities. Consider $\int \Sigma^i dA^i$. Integrating by parts we have
\be
\int \Sigma^i dA^i \approx  - \int d\Sigma^i A^i = -\int \epsilon^{ijk} A^i A^j \Sigma^k. 
\ee
In the last equality we have used (\ref{compat}). The quantity on the right-hand side is a multiple of $A J_1(A)$. This means that
\be
\int \Sigma^i F^i = \int \Sigma^i (dA^i + \frac{1}{2} \epsilon^{ijk} A^j A^k) \approx - \frac{1}{2} \int \Sigma^i \epsilon^{ijk} A^j A^k.
\ee
The integrand on the left is built from objects that transform covariantly under local ${\rm SO}(3)$ gauge transformations and is thus invariant. The integral is then both diffeomorphism and gauge invariant. This means that this is also the case for the object on the left-hand side. This establishes that there is unique action for ${\rm SU}(2)$ structures in dimension four that is both diffeomorphism and ${\rm SO}(3)$ gauge invariant. It is of the schematic type torsion squared, and is given by 
\be\label{second-order-action}
S[\Sigma] = - \frac{1}{2}\int_M \Sigma^i \epsilon^{ijk} A^j(\Sigma) A^k(\Sigma),
\ee
where we now indicated that the connection (intrinsic torsion) is determined by $\Sigma$. This is the action described in the Introduction, see (\ref{intr-action}). This functional is of the type of general functionals considered in \cite{Fowdar}, Sec. 3.3. Our discussion above makes it clear that the action (\ref{second-order-action}) is "the best" second order in derivatives action that can be written for ${\rm SU}(2)$ structures. It is the best action because it is the unique action that in addition to diffeomorphism invariance also possesses ${\rm SO}(3)$ gauge invariance. As we shall see below, it is also best in the sense that its critical points are Einstein. One can substitute the expression for $A(\Sigma)$ given by (\ref{A-dS}) to obtain an explicit functional in terms of $\Sigma$. 

\subsection{Plebanski action and Einstein condition}

We can now discuss the Plebanski action, which is a first-order in derivatives version of (\ref{second-order-action}). The idea is to write an action that is a functional of both $\Sigma^i$ and an independent $E$-valued one-form field $A^i$, such that the Euler-Lagrange equations for $A^i$ coincide with (\ref{compat}). A moment of reflection shows that this action is $\int \Sigma^i F^i$. This action is then to be supplemented by the constraint terms that guarantee that $\Sigma^i$ satisfy their algebraic constraints. One is also free to add to this action the volume term with an arbitrary coefficient. This produces the action  known in the literature as the Plebanski action 
\cite{Plebanski:1977zz}. It is given by
\be\label{pleb}
S[\Sigma,A,\Psi]= \int_M \Sigma^i ( dA^i + \frac{1}{2} \epsilon^{ijk} A^j A^k) - \frac{1}{2}\left( \Psi^{ij} + \frac{\Lambda}{3} \delta^{ij}\right) \Sigma^i \Sigma^j.
\ee
Here $\Psi^{ij}$ is an arbitrary traceless symmetric $3\times 3$ matrix, whose components serve as Lagrange multipliers to impose the constraints $\Sigma^i\Sigma^j\sim \delta^{ij}$. Indeed, the variation with respect to the field $\Psi^{ij}$ gives $\Sigma^i\Sigma^j\sim \delta^{ij}$, which are the algebraic conditions that need to be satisfied by an ${\rm SU}(2)$ structure defining 2-forms $\Sigma^i$. It is also easy to see that its Euler-Lagrange equation arising by varying with respect to $A^i$ is precisely (\ref{compat}), and the Euler-Lagrange equation arising by varying with respect to $\Sigma^i$ is precisely (\ref{einstein}).  

This establishes the following 
\begin{theorem} The critical points of (\ref{second-order-action}), or equivalently of (\ref{pleb}) are ${\rm SU}(2)$ structures whose associated metric is Einstein.
\end{theorem}

\section*{Acknowledgements} KK is grateful to Ilka Agrikola for the invitation to Marburg and discussions on the material presented here, and to Shubham Dwivedi for a discussion.

\end{document}